\documentclass{article}%
\usepackage{amsmath}
\usepackage{amsthm}
\usepackage{amsfonts}
\usepackage{amssymb}
\usepackage{graphicx}%
\usepackage{bbm}
\usepackage{color}
\usepackage{pgfplots}

\pgfplotsset{compat=1.18}

\newcommand{\mbR}{\mathbb{R}}
\newcommand{\E}{\mathsf{E}}
\newcommand{\1}{1\!\!\,{\rm I}}
\newcommand{\Pb}{\mathsf{P}}

\providecommand{\U}[1]{\protect\rule{.1in}{.1in}}
\theoremstyle{plain}
\newtheorem{theorem}{Theorem}

\newtheorem{corollary}[theorem]{Corollary}

\newtheorem{lemma}[theorem]{Lemma}

\newtheorem{remark}[theorem]{Remark}

\DeclareMathOperator{\sgn}{sgn}
\begin{document}

\title{The hard membrane process and transport barriers of turbulent flows}
%\author{Olga Aryasova, Franco Flandoli, Andrey Pilipenko}
\author{Olga Aryasova\footnote{Institute of Mathematics, Friedrich Schiller University Jena, Ernst--Abbe--Platz 2,
07743 Jena, Germany \emph{and} Institute of Geophysics, National Academy of Sciences of Ukraine, Palladin ave.\ 32, 03680 Kyiv-142,
Ukraine   \emph{and} Igor Sikorsky Kyiv Polytechnic Institute, Beresteiskyi ave.\ 37, 03056, Kyiv, Ukraine; \texttt{oaryasova@gmail.com}}, 
Franco Flandoli\footnote{Scuola normale superiore di Pisa, Piazza dei Cavalieri, 7, 56126 Pisa PI, Italy
\texttt{franco.flandoli@sns.it} },
Andrey Pilipenko
\footnote{University of Geneva, 
Section de math\'ematiques
UNI DUFOUR
24, rue du G\`en\`eral Dufour
Case postale 64
1211 Geneva 4, Switzerland \emph{and} Institute of Mathematics of  Ukrainian National Academy of Sciences
 Tereschenkivska str. 3, Kiev-4, 01601, Ukraine; \texttt{pilipenko.ay@gmail.com}}}

\maketitle
\section*{Abstract}

Motivated by the phenomenon of transport barriers in fusion plasma devices, we
write a mathematical model of heat dispersion in a turbulent fluid with a
transport barrier, properly idealized; in a scaling limit of the turbulence
model with separation of scales we get a heat equation with space-dependent
diffusion coefficient, poorly diffusing near the barrier; then we investigate
the scaling limit when the diffused barrier converges to a sharp separating
surface and describe the limit by means of the stochastic process called
Brownian motion with hard membrane. 

\section*{Keywords and classification}

heat equation; turbulence; transport barrier; Brownian motion with hard membrane

35K05, 76F45, 60J65

\section{Introduction}

\subsection{Physical motivation}

In confined fusion plasma devices like tokamaks, turbulence is always present
to some degree and it has the bad effect of dispersing heat and particles from
the central very hot and dense region to the boundary and then to the
surrounding areas and walls. Sometimes barriers arise in intermediate regions,
which reduce this dispersion. These barriers, also called zonal flows, are
thin layers of plasma where the fluid velocity is not turbulent as everywhere
else but ordered, roughly laminar, directed perpendicularly to the direction
of dispersion. See \cite{Diamond} for a review and \cite{Ding}, \cite{Garbet}
for more recent numerical simulations. 

There is still active research to understand how these barriers arise, whether
they can be triggered, how they evolve, the precise links with other plasma
dynamical features and parameters. Our aim here is not to contribute to these
difficult questions but only to rigorously define a mathematical model of a
sharp heat-diffusion barrier.

\subsection{Our aim}

The aim of this paper is to describe a sharp mathematical notion of
heat-diffusion barrier. More precisely, first we introduce the heat
diffusion-transport equation where the transport velocity field is a synthetic
turbulent fluid, a stochastic process with properties which are suitable for
mathematical investigation and relatively close to the reality of a turbulent
fluid with strong separation of scales. Very important, the velocity field
incorporates a transport barrier, in the sense that its turbulent features are
reduced and made laminar in a certain intermediate region.

Then we deduce, by a rigorous scaling limit theorem (Theorem 1 below), a heat
diffusion equation where the heat coefficient is space dependent and,
precisely, damped to zero in the neighborhood of an ideal surface (the surface
$x=0$ in the model below). It is very important that the space-dependent
coefficient appears in the heat operator in divergence form and this is
precisely motivated by the rigorous scaling limit theorem.

The model just obtained could be already considered a reasonable model of
diffusion with a transport barrier. In Section 2.1 we show by analytical
computations and by Figure 1 that the temperature profile corresponds to the
expectations:\ traditional diffusion between the domain boundaries and the
barrier, and a rapid temperature jump at the barrier. 

However, mathematically the object constructed until now is not a sharp
object. The barrier region is a neighbor of an ideal surface. We think it is
conceptually very interesting to understand the limit when this region becomes
a true surface. The barrier then, in the appropriate limit, is the separating
surface. We reach this goal by means of hard membrane. A probabilistic construction of {\it the Brownian motion with hard membrane} or the so called {\it snapping out Brownian motion (SNOB)} was introduced  in \cite{Lejay, ManPil} and can be described as follows. Assume that a (one-dimensional) process starts at the positive half-line and behaves as a  Brownian motion reflected at 0 into the positive half-line, until its local time at $0$  reaches a certain exponentially distributed random variable. At this instant, the SNOB process ``switches orientation from positive to negative'' and behaves as  a  Brownian motion reflected at 0 into the negative half-line, again  until its local time reaches another exponentially distributed random variable. Then it switches back to the positive orientation, reflects at 0 upward, and so on. It is assumed that the behavior of SNOB after each switch is independent of the past, and all exponential random variables are independent of each other  and also of the driving Brownian motion. The point 0 can be considered as a tight membrane that prevent particles from passing through in most collisions.  

The snapping out Brownian motion arises naturally    in \cite{ Lejay, ManPil} as a limit of a Brownian motions with a potential   localized in a small neighborhood of 0, which  strongly pushes the diffusion away from 0. Interestingly, the semigroup associated with SNOB corresponds to Newton's law of colling.   The semigroup approach for SNOB was introduced in \cite{BobrowskiSNOB1, BobrowskiSNOB2}. See also some applications of the process in chemistry \cite{ANDREWS, Tanner}, magnetic resonance imaging \cite{Fieremans},  and in modelling inhibitory synaptic receptor dynamic \cite{Bressloff}.

\subsection{The stochastic fluid model}%

%TCIMACRO{\FRAME{itbpF}{2.6493in}{1.6069in}{0in}{}{}{Figure}%
%maintain-aspect-ratio TRUE;  display "USEDEF";  valid_file "T";
%width 2.6493in;  height 1.6069in;  depth 0in;  original-width 6.8668in;
%original-height 4.1483in;  cropleft "0";  croptop "1";  cropright "1";
%cropbottom "0";  tempfilename 'SVQ3E800.wmf';tempfile-properties "XPR";}} }%
%BeginExpansion
%{\includegraphics[
%natheight=4.148300in,
%natwidth=6.866800in,
%height=1.6069in,
%width=2.6493in
%]%
%{SVQ3E800.wmf}%
%}
%EndExpansion

Our idealized model is composed by the domain $D=\left[  -1,1\right]
\times\left[  0,2\pi\right]  $ with periodic conditions in the vertical
direction. The fluid occupies this domain and is given by a velocity field
$\mathbf{u}\left(  x,y,t\right)  $, $x\in\left[  -1,1\right]  $, $y\in\left[
0,2\pi\right]  $. In this section, we specify the properties of the velocity field, while the transport of heat in $D$
is considered in Section \ref{section:heat_equation} 
We assume that $\mathbf{u}\left(  x,y,t\right)$ satisfies the following conditions:
\begin{enumerate}
\item[C1)]
\[
\operatorname{div}\mathbf{u}\left(  x,y,t\right)  =0,
\]
\item[C2)]
periodic conditions in the vertical direction, 
\item[C3)] slip boundary condition at
the boundaries $x=\pm1$:%
\[
\mathbf{u}\left(  \pm1,y,t\right)  \cdot\mathbf{e}_{1}=0,
\]
where $\mathbf{e}_{1}=\left(  1,0\right)  $. 
\end{enumerate}
Moreover, we shall tune
$\mathbf{u}\left(  x,y,t\right)  $ so that it produces a transport barrier
along the vertical line $x=0$:
\begin{enumerate}
\item[C4)]
transport barrier around $x=0$.
\end{enumerate}

The fluid will have the form%
\[
\mathbf{u}\left(  x,y,t\right)  =\nabla^{\perp}\psi\left(  x,y,t\right),
\]
where $\nabla^{\perp}=\left(  -\partial_{y},\partial_{x}\right)  $ and
$\psi\left(  x,y,t\right)  $ is a given stream function, periodic in the
vertical direction and with constant values $\psi_{\pm1}$ at the boundaries
$x=\pm1$%
\[
\psi\left(  \pm1,y,t\right)  =\psi_{\pm1}\text{ for all }y,t.
\]
The conditions C1--C3 on $\mathbf{u}$\ are satisfied; in particular, since the
boundaries $x=\pm1$ are level curves of $\psi$, $\nabla\psi$ is directed as
$\mathbf{e}_{1}$, and $\mathbf{u}=\nabla^{\perp}\psi$ is orthogonal to
$\mathbf{e}_{1}$.

The stream function $\psi\left(  x,y,t\right)  =\psi_{N,\epsilon}\left(
x,y,t\right)  $ is parametrized by $N\in\mathbb{N}$ and $\epsilon\in\left(
0,1\right)  $ and given by
\[
\psi_{N,\epsilon}\left(  x,y,t\right)  =\sqrt{\kappa_{T}^{\epsilon}}%
\sum_{\mathbf{k}\in\mathbb{Z}_{0}^{2}}C\left(  \mathbf{k},N\right)
\psi_{\mathbf{k}}^{\epsilon}\left(  x,y\right)  \overset{\cdot
}{W}_{t}^{\mathbf{k}},
\]
where $\kappa_{T}^{\epsilon}>0$ is a parameter describing the strength of the  
turbulent diffusion, $\mathbb{Z}_{0}^{2}$ is defined below, $\left(  W_{t}^{\mathbf{k}}\right)  _{\mathbf{k}\in
\mathbb{Z}_{0}^{2}}$ is a family of independent Brownian motions,
$\psi_{\mathbf{k}}^{\epsilon}\left(  x,y\right)  $ are functions parametrized
by $\mathbf{k}\in\mathbb{Z}_{0}^{2}$ and $\epsilon\in\left(  0,1\right)  $
specified below and satisfying the boundary conditions mentioned above for
$\psi\left(  x,y,t\right)  $ (hence $\psi\left(  x,y,t\right)  $ satisfies
such boundary conditions), $C\left(  \mathbf{k},N\right)  $ are suitable
constants depending on $\mathbf{k}\in\mathbb{Z}_{0}^{2}$ and $N\in\mathbb{N}$
such that, for every $N\in\mathbb{N}$, the coefficients $\left(  C\left(
\mathbf{k},N\right)  \right)  _{\mathbf{k}\in\mathbb{Z}_{0}^{2}}$ are equal to
zero except for a finite number of indexes $\mathbf{k}$; hence the sum in the
equation is always finite.

For every $\epsilon\in\left(  0,1\right)  $, let $\chi_{\epsilon}\left(
x\right)  $ be a smooth function defined for $x\in\mathbb{R}$, equal to 1 in
$\left(  -2\epsilon,2\epsilon\right)  ^{c}$, such that $\chi_\epsilon\left(  \left(
-2\epsilon,2\epsilon\right)  \right)  \subset\left[  0,1\right]  $ and
$\chi_\epsilon\left(  0\right)  =0$. We shall use this function to define the barrier.
We shall make a very precise choice below, motivated by scaling limit
arguments, but let us keep it less determined for the time being.%

%TCIMACRO{\FRAME{itbpFU}{2.437in}{1.2998in}{0in}{\Qcb{Shape of $\chi_{\epsilon
%}$}}{}{Figure}{\special{ language "Scientific Word";  type "GRAPHIC";
%maintain-aspect-ratio TRUE;  display "USEDEF";  valid_file "T";
%width 2.437in;  height 1.2998in;  depth 0in;  original-width 6.1947in;
%original-height 3.2846in;  cropleft "0";  croptop "1";  cropright "1";
%cropbottom "0";  tempfilename 'SVQ9LC02.wmf';tempfile-properties "XPR";}} }%
%BeginExpansion
%{\parbox[b]{2.437in}{\begin{center}
%\includegraphics[
%natheight=3.284600in,
%natwidth=6.194700in,
%height=1.2998in,
%width=2.437in
%]%
%{SVQ9LC02.wmf}%
%\\
%Shape of $\chi_{\epsilon}$%
%\end{center}}}
%EndExpansion

Let $\chi_{\epsilon,bdr}\left(  x\right)  $ be a smooth function  with exactly the same properties as  $\chi_{\epsilon}\left(  x\right)  $ (but not
necessarily equal to $\chi_{\epsilon
}\left(  x\right)  $). We shall use this function to define the boundary layer
of the fluid at $x=\pm1$. This boundary layer does not play any role for our
purposes, but we have to include it for coherence, since the fluid should not
penetrate the lateral boundaries.

Call $\mathbb{Z}_{++}^{2}$ the set of all $\mathbf{k}=\left(  k_{1}%
,k_{2}\right)  \in{(\pi\mathbb{Z})\times\mathbb{Z}}$ such that $k_{1}>0, k_{2}>0$.
Similarly, $\mathbb{Z}_{+-}^{2}$ is defined by $k_{1}>0, k_{2}<0$; $\mathbb{Z}_{-+}^{2}$  by $k_{1}<0, k_{2}>0$; $\mathbb{Z}_{--}^{2}$  by $k_{1}<0, k_{2}<0$. Finally, set $\mathbb{Z}_{+}^{2}=\mathbb{Z}_{++}^{2}\cup \mathbb{Z}_{+-}^{2}$, $\mathbb{Z}_{-}^{2}=\mathbb{Z}_{-+}^{2}\cup \mathbb{Z}_{--}^{2}$, $\mathbb{Z}_{0}^{2}=\mathbb{Z}_{+}^{2}\cup \mathbb{Z}_{-}^{2}$.

Set (with the notation $\mathbf{x}=\left(  x,y\right)  $)%
\begin{align*}
\psi_{\mathbf{k}}^{\epsilon}\left(  \mathbf{x}\right)   &  =\frac{1}{\left\vert
\mathbf{k}\right\vert }\chi_{\epsilon,bdr}\left(  x-1\right)  \chi
_{\epsilon,bdr}\left(  x+1\right)  \sin\left(  \mathbf{k}\cdot\mathbf{x}%
\right)  \chi_{\epsilon}\left(  x\right)  \qquad\text{for }\mathbf{k}%
\in\mathbb{Z}_{+}^{2},\\
\psi_{\mathbf{k}}^{\epsilon}\left(  \mathbf{x}\right)   &  =\frac{1}{\left\vert
\mathbf{k}\right\vert }\chi_{\epsilon,bdr}\left(  x-1\right)  \chi
_{\epsilon,bdr}\left(  x+1\right)  \cos\left(  \mathbf{k}\cdot\mathbf{x}%
\right)  \chi_{\epsilon}\left(  x\right)  \qquad\text{for }\mathbf{k}%
\in\mathbb{Z}_{-}^{2}.%
\end{align*}
Here and below, 
$\left\vert
\mathbf{k}\right\vert$  denotes the Euclidean norm of 
$\mathbf{k}$, and 
$\mathbf{k}\cdot \mathbf{x}$ denotes the scalar product of 
$\mathbf{k}$ and $\mathbf{x}$.
These functions satisfy the boundary conditions, since all factors are
periodic in the vertical direction and the factor $\chi_{\epsilon,bdr}\left(
x-1\right)  \chi_{\epsilon,bdr}\left(  x+1\right)  $ vanishes at $x=\pm1$
(recall that $\chi_{\epsilon, bdr}\left(  0\right)  =0$). 
Therefore we have
\[
\mathbf{u}_{N,\epsilon}\left(  \mathbf{x},t\right)  =\sqrt{\kappa_{T}^{\epsilon}}%
\sum_{\mathbf{k}\in\mathbb{Z}_{0}^{2}}C\left(  \mathbf{k},N\right)
\mathbf{\sigma}_{\mathbf{k}}^{\epsilon}\left(  \mathbf{x}\right)  \overset{\cdot
}{W}_{t}^{\mathbf{k}},
\]
where, with the notation %
\[
\overline{\chi}_{\epsilon}\left(  x\right)  =\chi_{\epsilon, bdr}\left(
x-1\right)  \chi_{\epsilon, bdr}\left(  x+1\right)  \chi_{\epsilon}\left(
x\right),
\]
\begin{align*}
\sigma_{\mathbf{k}}^{\epsilon}\left(  \mathbf{x}\right)   &  =\frac{1}{\left\vert
\mathbf{k}\right\vert }\nabla^{\perp}\left[  \sin\left(  \mathbf{k}%
\cdot\mathbf{x}\right)  \overline{\chi}_{\epsilon}\left(  x\right)  \right]
\qquad\text{for }\mathbf{k}\in\mathbb{Z}_{+}^{2},\\
\sigma_{\mathbf{k}}^{\epsilon}\left(  \mathbf{x}\right)   &  =\frac{1}{\left\vert
\mathbf{k}\right\vert }\nabla^{\perp}\left[  \cos\left(  \mathbf{k}%
\cdot\mathbf{x}\right)  \overline{\chi}_{\epsilon}\left(  x\right)  \right]
\qquad\text{for }\mathbf{k}\in\mathbb{Z}_{-}^{2}.
\end{align*}

Let us interpret this model in terms of turbulence and transport barriers. The
terms $\frac{1}{\left\vert \mathbf{k}\right\vert }\sin\left(  \mathbf{k}%
\cdot\mathbf{x}\right)  $ and $\frac{1}{\left\vert \mathbf{k}\right\vert }%
\cos\left(  \mathbf{k}\cdot\mathbf{x}\right)  $ are a typical choice (see
\cite{Galeati}, or Chapters 3 and 4 of \cite{FlaLuongo} and references
therein) for the description of a turbulent fluid. There are of course other
choices and our results, in principle, do not depend on this factor of the
model, but the choice above is particularly good for the simplicity of the
final analytic expressions.

The turbulent fluid, however, should satisfy the slip boundary conditions at
$x=\pm1$, and this is caused by the cut-off function $\chi_{\epsilon
,bdr}\left(  x-1\right)  \chi_{\epsilon,bdr}\left(  x+1\right)  $.

Finally, essental for our model, there is a barrier at $x=0$. The barrier is
defined by a depletion of the vortical
structures produced by the combination of the functions $\frac{1}{\left\vert
\mathbf{k}\right\vert }\sin\left(  \mathbf{k}\cdot\mathbf{x}\right)  $ and
$\frac{1}{\left\vert \mathbf{k}\right\vert }\cos\left(  \mathbf{k}%
\cdot\mathbf{x}\right)  $. This depletion is given by the factor
$\chi_{\epsilon}\left(  x\right)  $. 

We shall be more specific in the choice of the turbulent part. The intuition
is that we model a fluid with small-scale (rapidly varying) turbulence in the
full domain (except a small boundary layer at $x=\pm1$ and except for the
transport barrier at $x=0$). Hence we introduce the set of indexes
\[
K_{N}=\left\{  \mathbf{k}\in\mathbb{Z}_{0}^{2}: {\sqrt{\left(\frac{k_1}{\pi}\right)^2+k_2^2}}%
\in[N,2N]\right\}
\]
and we limit the sum defining $\mathbf{u}_{N,\epsilon}\left(  x,y,t\right)  $
by taking
\begin{align*}
C\left(  \mathbf{k},N\right)   &  =0\text{ for }\mathbf{k}\in K_{N}^{c},\\
C\left(  \mathbf{k},N\right)   &  =\frac{1}{c_{N}}\text{ for }\mathbf{k}\in
K_{N}.
\end{align*}
Moreover, as we shall see below, in order to normalize the sum in Section 2, we choose
\[
c_{N}=\sqrt{\frac{Card\left(  K_{N}\cap\mathbb{Z}_{++}^{2}\right)  }{2}}.%
\]
Asymptotically as
$N\rightarrow\infty$,
\[
Card\left(  K_{N}\cap\mathbb{Z}_{++}^{2}\right)  \sim\frac{1}{4}\left[
4\pi N^{2}-\pi N^{2}\right]  =\frac34\pi{N^{2}}.
\]

\subsection{The heat equation}\label{section:heat_equation}

We consider diffusion and transport of heat in $D=\left[  -1,1\right]
\times\left[  0,2\pi\right]  $ according to the equation%

\begin{align*}
\partial_{t}T_{N,\epsilon}+\mathbf{u}_{N,\epsilon}\left(  x,y,t\right)
\cdot\nabla T_{N,\epsilon}  &  =\kappa_{\epsilon}\Delta T_{N,\epsilon},\\
T_{N,\epsilon}\left(  -1,y,t\right)   &  =T_{+}>0\text{ for all }y,t,\\
T_{N,\epsilon}\left(  1,y,t\right)   &  =0\text{ for all }y,t
\end{align*}
with a given initial condition $\theta_{0}\left(  x,y\right)  $, which however will
play only a transitory role, since we are mostly interested in the stationary
regime. The boundary condition $T_{N,\epsilon}\left(  -1,y,t\right)  =T_{+}$
corresponds to the fact that heat is continuouly supplied at $x=-1$, keeping
constant the temperature; similarly, heat is subtracted to the system at
$x=1$. Therefore, there is a continuous flux of heat from $x=-1$ to $x=1$.

Heat diffuses with the very small molecular diffusion constant $\kappa
_{\epsilon}>0$ and it is transported by the turbulent fluid velocity
$\mathbf{u}_{N,\epsilon}$ given above, incorporating a transport barrier at
$x=0$.

Equation above is a Stochastic Partial Differential Equation (SPDE), more
properly written in the form%
\[
dT_{N,\epsilon}+\sqrt{\kappa_{T}^{\epsilon}}\sum_{\mathbf{k}\in\mathbb{Z}%
_{0}^{2}}C\left(  \mathbf{k},N\right)  \mathbf{\sigma}_{\mathbf{k}}^{\epsilon
}\left(  x,y\right)  \cdot\nabla T_{N,\epsilon}\circ dW_{t}^{\mathbf{k}%
}=\kappa
_{\epsilon}\Delta T_{N,\epsilon}dt,
\]
where we adopt the Stratonovich integration, as in all the analogous
investigations (see \cite{FlaLuongo} and references therein), to preserve
conservation properties and Wong Zakai approximation meaning. Moreover, at the
rigorous level, to avoid working with Stratonovich integrals (which require
more detaled martingale properties of solutions, possible but annoynig), we
shall formulate the equation in It\^{o} form%
\begin{equation}
dT_{N,\epsilon}+\sqrt{\kappa_{T}^{\epsilon}}\sum_{\mathbf{k}\in\mathbb{Z}%
_{0}^{2}}C\left(  \mathbf{k},N\right)  \mathbf{\sigma}_{\mathbf{k}}^{\epsilon
}\left(  x,y\right)  \cdot\nabla T_{N,\epsilon}dW_{t}^{\mathbf{k}}%
=\kappa
_{\epsilon}\Delta T_{N,\epsilon}dt+\kappa_{T}^{\epsilon}\mathcal{L}%
_{N,\epsilon}T_{N,\epsilon}dt,\label{SPDE Ito form}%
\end{equation}
where $\mathcal{L}_{N,\epsilon}T_{N,\epsilon}dt$ is the It\^{o}-Stratonovich
corrector, discussed below.

We want to study, in sequence, two limits:

\begin{enumerate}
\item First, the limit as $N\rightarrow\infty$, the so-called diffusion limit.
It will generate the deterministic model%
\begin{align}
\partial_{t}T_{\epsilon} &  =\kappa_{\epsilon}\Delta T_{\epsilon}+\kappa_{T}^{\epsilon
}\mathcal{L}_{\epsilon}T_{\epsilon}, \label{eq:Teps}\\
T_{\epsilon}\left(  -1,y,t\right)   &  =T_{+}>0\text{ for all }y,t, \label{eq:Teps1}\\
T_{\epsilon}\left(  1,y,t\right)   &  =0\text{ for all }y,t,\label{eq:Teps2}
\end{align}
where $\mathcal{L}_{\epsilon}$ is some limit operator.

\item Second, the limit of the boundary problem \eqref{eq:Teps}-\eqref{eq:Teps2} as $\epsilon\rightarrow0$. This second limit is the
main contribution of the paper.
\end{enumerate}

In Section \ref{Section diffusion limit} we discuss the limit problem 1, which is relatively classical but necessary to motivate problem 2, and the specific form of the operator
 $\mathcal{L}_{\epsilon}$. Then, in Section
\ref{Section hard membrane}, we discuss the limit problem 2.%

%TCIMACRO{\FRAME{itbpFU}{2.4436in}{1.2619in}{0in}{\Qcb{Expected shape of
%$T_{\epsilon}$}}{}{Figure}{\special{ language "Scientific Word";
%type "GRAPHIC";  maintain-aspect-ratio TRUE;  display "USEDEF";
%valid_file "T";  width 2.4436in;  height 1.2619in;  depth 0in;
%original-width 7.0551in;  original-height 3.6253in;  cropleft "0";
%tempfilename 'SVQ9KP01.wmf';tempfile-properties "XPR";}} }%
%BeginExpansion
%{\parbox[b]{2.4436in}{\begin{center}
%\includegraphics[
%natheight=3.625300in,
%natwidth=7.055100in,
%height=1.2619in,
%width=2.4436in
%]%
%{SVQ9KP01.wmf}%
%\\
%Expected shape of $T_{\epsilon}$%
%\end{center}}}
%EndExpansion

\section{The diffusion limit\label{Section diffusion limit}}

For existence, uniqueness and regularity results on equation
(\ref{SPDE Ito form}), and for more details on the It\^{o}-Stratonovich
corrector, one can see several works, including \cite{FlaLuongo}, Chapter 3.
In particular, let us recall the following result. If the initial condition
$\theta_0$ is of class $L^{2}\left(  D\right)  $, and $\kappa_{\epsilon}>0$,
there exists a unique solution, continuous adapted process in $L^{2}\left(
D\right)  $, with trajectories of class
\[
T_{N,\epsilon}\in C\left(  \left[  0,T\right]  ;L^{2}\left(  D\right)
\right)  \cap L^{2}\left(  0,T;W^{1,2}\left(  D\right)  \right)
\]
for all $T>0$. Moreover, if $\theta_{0}\geq0$ then also the solution is non
negative for all $t>0$ and if $\theta_0\leq C$ the same bound holds for the
solution. When $\kappa_{\epsilon}=0$, we still have existence and uniqueness
in the class of continuous adapted process in $L^{2}\left(  D\right)  $,
missing only the property $L^{2}\left(  0,T;W^{1,2}\left(  D\right)  \right)
$. Most of the literature on this topic deals with homogeneous boundary conditions. 
To deal with non-homogeneous ones, see the proof of Theorem \ref{Thm diffusion limit} below.

Let us reconstruct in our particular case the computation of the
It\^{o}-Stratonovich corrector. Let us recall that, when a Stratonovich term
is linear in the state function (here $T$), of the form $\sum_{k}B_{k}T\circ
dW_{t}^{k}$ for certain linear operators $B_{k}$, the It\^{o}-Stratonovich
corrector is $\frac{1}{2}\sum_{k}B_{k}B_{k}Tdt$; in our case, since
$
\frac12\operatorname{div}\mathbf{\sigma}_{\mathbf{k}}^{\epsilon}=0$, it is easy to
recognize that $\frac{1}{2}\sum_{k}B_{k}B_{k}T$ has the form $\kappa_T^\epsilon\mathcal{L}_{N,\epsilon} T$, where
\begin{equation}\label{eq:L}
\left(  \mathcal{L}_{N,\epsilon}f\right)  \left(  \mathbf{x}\right)
=\operatorname{div}\left(  Q_{N,\epsilon}\left(  \mathbf{x},\mathbf{x}\right)
\nabla f\right)
\end{equation}
with
\begin{align*}
Q_{N,\epsilon}\left(  \mathbf{x},\mathbf{x}^{\prime}\right)   &  =\frac
{1}{2 c_{N}^{2}}\sum_{\mathbf{k}\in K_{N}}\mathbf{\sigma}_{\mathbf{k}}%
^{\epsilon}\left(  \mathbf{x}\right)  \otimes\mathbf{\sigma}_{\mathbf{k}%
}^{\epsilon}\left(  \mathbf{x}^{\prime}\right),  \\
\mathbf{x} &  =\left(  x,y\right)  ,\mathbf{x}^{\prime}=\left(  x^{\prime
},y^{\prime}\right)  \text{ in }D.
\end{align*}
Let us compute more precisely this matrix-valued function $Q_{N,\epsilon
}\left(  \mathbf{x},\mathbf{x^\prime}\right)  $. Then we have (always with the
notations $\mathbf{x}=\left(  x,y\right)  ,\mathbf{x}^{\prime}=\left(
x^{\prime},y^{\prime}\right)  $) 
\begin{align*}
Q_{N,\epsilon}\left(  \mathbf{x},\mathbf{x}^{\prime}\right)   &  =\frac12\frac
{1}{c_{N}^{2}}\sum_{\mathbf{k}\in K_{N}\cap \mathbb{Z}_+^2}\frac{1}{\left\vert \mathbf{k}%
\right\vert ^{2}}\nabla^{\perp}\left[  \sin\left(  \mathbf{k}\cdot
\mathbf{x}\right)  \overline{\chi}_{\epsilon}\left(  x\right)  \right]
\otimes\nabla^{\perp}\left[  \sin\left(  \mathbf{k}\cdot\mathbf{x}^{\prime
}\right)  \overline{\chi}_{\epsilon}\left(  x^{\prime}\right)  \right]  \\
&  +\frac12\frac{1}{c_{N}^{2}}\sum_{\mathbf{k}\in K_{N}\cap \mathbb{Z}_-^2}\frac{1}{\left\vert
\mathbf{k}\right\vert ^{2}}\nabla^{\perp}\left[  \cos\left(  \mathbf{k}%
\cdot\mathbf{x}\right)  \overline{\chi}_{\epsilon}\left(  x\right)  \right]
\otimes\nabla^{\perp}\left[  \cos\left(  \mathbf{k}\cdot\mathbf{x}^{\prime
}\right)  \overline{\chi}_{\epsilon}\left(  x^{\prime}\right)  \right].
\end{align*}

Decomposing the derivatives $\nabla^{\perp}$ of products, we get several
terms, precisely 8 terms, which are defined by formulas \eqref{eq:Q1}-\eqref{eq:Q8}  in the proof of Theorem \ref{Thm diffusion limit} below:%
\[
Q_{N,\epsilon}\left(  \mathbf{x},\mathbf{x}^{\prime}\right)  =Q_{N,\epsilon
}^{\left(  1\right)  }\left(  \mathbf{x},\mathbf{x}^{\prime}\right)
+....+Q_{N,\epsilon}^{\left(  8\right)  }\left(  \mathbf{x},\mathbf{x}%
^{\prime}\right)  .
\]
The main term, computed on the diagonal, is $Q^{(1)}+Q^{(5)}$  equal to 
\begin{align*}
\overline {Q}_{N,\epsilon}\left(  \mathbf{x},\mathbf{x}\right)   &
=\frac12\frac{1}{c_{N}^{2}}\sum_{\mathbf{k}\in K_{N}\cap \mathbb{Z}_+^2}\frac{\mathbf{k}^{\perp}%
\otimes\mathbf{k}^{\perp}}{\left\vert \mathbf{k}\right\vert ^{2}}%
\overline{\chi}_{\epsilon}^{2}\left(  x\right)  \cos^{2}\left(  \mathbf{k}%
\cdot\mathbf{x}\right)  \\
&  +\frac12\frac{1}{c_{N}^{2}}\sum_{\mathbf{k}\in K_{N}\cap \mathbb{Z}_-^2}\frac{\mathbf{k}^{\perp
}\otimes\mathbf{k}^{\perp}}{\left\vert \mathbf{k}\right\vert ^{2}}%
\overline{\chi}_{\epsilon}^{2}\left(  x\right)  \sin^{2}\left(  \mathbf{k}%
\cdot\mathbf{x}\right),
\end{align*}
where $\mathbf{k}^{\perp}=\left(  -k_{2},k_{1}\right)  $, which is equal to
\[
\frac{\overline{\chi}_{\epsilon}\left(  x\right)  ^{2}}{2c_{N}^{2}}%
\sum_{\mathbf{k}\in K_{N}\cap \mathbb{Z}_+^2}\frac{\mathbf{k}^{\perp}\otimes\mathbf{k}^{\perp}%
}{\left\vert \mathbf{k}\right\vert ^{2}}.
\]
Moreover, splitting $\mathbb{Z}_{+}^{2}$ in the set $\mathbb{Z}_{++}^{2}$ and 
$\mathbb{Z}_{+-}^{2}$ as said above, it is easy to see that
\[
\sum_{\mathbf{k}\in K_{N}\cap \mathbb{Z}_+^2}\frac{\mathbf{k}^{\perp}\otimes\mathbf{k}^{\perp}%
}{\left\vert \mathbf{k}\right\vert ^{2}}=Card\left(  K_{N}\cap\mathbb{Z}% 
_{++}^{2}\right)  I_{2},%
\]
where $I_{2}$ is the identity matrix in two dimensions. Therefore, recalling
the definition of $c_{N}$,
\begin{align*}
\overline{Q}_{N,\epsilon}\left(  \mathbf{x},\mathbf{x}\right)   &
=\frac{\overline{\chi}_{\epsilon}\left(  x\right)  ^{2}}{2c_{N}^{2}}Card\left(
K_{N}\cap\mathbb{Z}_{++}^{2}\right)  I_{2}\\
&  =\overline{\chi}_{\epsilon}\left(  x\right)  ^{2}I_{2}%
\end{align*}
so that
\[
\operatorname{div}\left(  \overline{Q}_{N,\epsilon}\left(  \mathbf{x},\mathbf{x}\right)  \nabla T_{N,\epsilon}\right)
=\operatorname{div}\left(  \overline{\chi}_{\epsilon}^{2}\left(  x\right)
\nabla T_{N,\epsilon}\right)  .
\]

The other terms $Q^{(2)}, Q^{(3)}, Q^{(4)}, Q^{(6)}, Q^{(7)}, Q^{(8)}$ contain a lot of terms, but they always miss
some derivatives of $\sin\left(  \mathbf{k}\cdot\mathbf{x}\right)  $,
$\cos\left(  \mathbf{k}\cdot\mathbf{x}\right)  $. Therefore they keep at least
one factor $\frac{1}{\left\vert \mathbf{k}\right\vert }$ inside the sum,
opposite to the first terms $Q^{(1)}$ and $Q^{(5)}$, where all factors $\frac{1}{\left\vert \mathbf{k}\right\vert }$
cancel out due to the derivatives of sine and cosine. The consequence is that
\[
\lim_{N\rightarrow\infty}Q_{N,\epsilon}\left(  \mathbf{x},\mathbf{x}\right)
=\lim_{N\rightarrow\infty}Q_{N,\epsilon}^{\left(  1\right)  }\left(
\mathbf{x},\mathbf{x}\right)  =2\overline{\chi}_{\epsilon}^{2}\left(
x\right)  I_{2}%
\]
and, in a suitable sense (e.g. of distributions), the operator $\frac{1}%
{2}\operatorname{div}\left(  Q_{N,\epsilon}\left(  \mathbf{x},\mathbf{x}%
\right)  \nabla f\right)  $ converges to $\operatorname{div}\left(
\overline{\chi}_{\epsilon}^{2}\left(  x\right)  \nabla f\right)  $, as
$N\rightarrow\infty$. 

The basic line of proof of the following theorem is classical, see, for instance, \cite{FlaLuongo}, Chapter 3, or the original paper \cite{Galeati}. {However, some details are new and more difficult, hence we give the proof in
Appendix, where we also postpone the definition of the function space
$\mathcal{D}\left( A_{\epsilon}^{2}\right)  $, a dense subspace of $L^{2}\left(
D\right)  $, see Step 0 of the Appendix.}
\begin{theorem}
\label{Thm diffusion limit} {Let $\theta_{0}\in L^{2}\left(  D\right)  $ and $\overline{\chi}_{\epsilon}$
be a smooth function. For every test function $\phi\in \mathcal{D}\left( A_{\epsilon
}^{2}\right)  $,} 
\[
\lim_{N\rightarrow\infty}\mathbb{E}\left[  \left(  \left\langle T_{N,\epsilon
}\left(  t\right)  ,\phi\right\rangle -\left\langle{T}_{\epsilon
}\left(  t\right)  ,\phi\right\rangle \right)  ^{2}\right]  =0,
\]
where $T_{\epsilon}\left(  x,y,t\right)  $ is the solution on $D$ of the
equation%
\begin{align}
\label{eq:theor_1}
\partial_{t}T_{\epsilon}\left(  x,y,t\right)  
&  =\kappa_{\epsilon}\Delta T_{\epsilon}\left(  x,y,t\right)  +\kappa
_{T}^{\epsilon}\operatorname{div}\left(  \overline{\chi}_{\epsilon}^{2}\left(
x\right)  \nabla T_{\epsilon}\left(  x,y,t\right)  \right),  \\
T_{\epsilon}\left(  -1,y,t\right)   &  =T_{+}>0\text{ for all }y,t  \label{bound_cond_left},\\
T_{\epsilon}\left(  1,y,t\right)   &  =0\text{ for all }y,t  \label{bound_cond_right},\\
\label{eq:theor_1_init}
T_{\epsilon}|_{t=0} &  =\theta_0%
\end{align}
with periodic boundary conditions in the vertical direction.
\end{theorem}

\subsection{Stationary solution of the diffusion limit equation}

The stationary problem, for the unknown function $\overline{T}_{\epsilon
}:D\rightarrow\mathbb{R}$, associated to the diffusion limit equation of the
previous theorem is%
\begin{align*}
\kappa_{\epsilon}\Delta\overline{T}_{\epsilon}+\kappa_{T}^{\epsilon
}\operatorname{div}\left(  \overline{\chi}_{\epsilon}^{2}\nabla\overline
{T}_{\epsilon}\right)   &  =0\\
\overline{T}_{\epsilon}\left(  -1,y\right)   &  =T_{+}>0\text{ for all }y\\
\overline{T}_{\epsilon}\left(  1,y\right)   &  =0\text{ for all }y
\end{align*}
with periodicity in the $y$ direction. Recall that the function $\overline{\chi}_\epsilon$ depends only on $x$. The solution is constant in $y$; hence it has the form $\overline{T}_{\epsilon}\left(  x\right)  $
 and can be computed analytically: for some constant $\phi_{q}^{\epsilon}$ (interpreted below) we
have%
\[
\left(  \kappa_{\epsilon}+\kappa_{T}^{\epsilon}\overline{\chi}_{\epsilon}%
^{2}\left(  x\right)  \right)  \overline{T}_{\epsilon}^{\prime}\left(
x\right)  =-\phi_{q}^{\epsilon}.
\]
It follows (recall that $\overline{\chi}_{\epsilon}^{2}\left(  \pm1\right)
=0$) that $\kappa_{\epsilon}\overline{T}_{\epsilon}^{\prime}\left(  -1\right)
=-\phi_{q}^{\epsilon}$ and $\kappa_{\epsilon}\overline{T}_{\epsilon}^{\prime
}\left(  1\right)  =-\phi_{q}^{\epsilon}$, hence also $\overline{T}_{\epsilon
}^{\prime}\left(  1\right)  =\overline{T}_{\epsilon}^{\prime}\left(
-1\right)  $. Then the stationary temperature is equal to%
\[
\overline{T}_{\epsilon}\left(  x\right)  =T_{+}-\int_{-1}^{x}\frac{\phi
_{q}^{\epsilon}}{\kappa_{\epsilon}+\kappa_{T}^{\epsilon}\overline{\chi
}_{\epsilon}^{2}\left(  x^{\prime}\right)  }dx^{\prime}.%
\]
From the boundary condition
$\overline{T}_{\epsilon}\left(  1\right)  =0$, we get%
\[
\phi_{q}^{\epsilon}=\frac{T_{+}}{\int_{-1}^{1}\frac{1}{\kappa_{\epsilon
}+\kappa_{T}^{\epsilon}\overline{\chi}_{\epsilon}^{2}\left(  x^{\prime
}\right)  }dx^{\prime}}.
\]

Consider, for graphical purposes, the following example (the function
$\overline{\chi}_{\epsilon}^{2}\left(  x\right)  $ deviates minimally from our
assumptions but just for numerical simplicity):%
\begin{equation}
\overline{\chi}_{\epsilon}^{2}\left(  x\right)  =\left(  \left(  \frac{2}{\pi
}\right)  ^{3}\arctan\left(  \frac{x}{\epsilon}\right)  \arctan\left(
10^{4}\left(  x-1\right)  \right)  \arctan\left(  10^{4}\left(  x+1\right)
\right)  \right)  ^{2}\label{example}.%
\end{equation}
The graph of $\overline{T}_{\epsilon}\left(  x\right) $ with $ \epsilon=0.2$, $\ T_{+}=2$, $\ \kappa_{T}^{\epsilon}=0.1$, $
\ \kappa_{\epsilon}=0.004$ is the magenta
profile in Figure 1. The main purpose of this paper, after having justified
the limit problem above by a turbulent fluid model, is to investigate the
limit as $\epsilon\rightarrow0$ of this temperature profile, from the
viewpoint of its probabilistic interpretation.%

%TCIMACRO{\FRAME{ftbpFU}{3.1327in}{3.1262in}{0pt}{\Qcb{Temperature profiles
%$\overline{T}\left(  x\right)  $ in three cases: $\epsilon=0.2$ (sky blue),
%$\epsilon=0.01$ (blue), both based on equation (\ref{example}), and the
%profile without barrier (grey).}}{}{Figure}%
%maintain-aspect-ratio TRUE;  display "USEDEF";  valid_file "T";
%width 3.1327in;  height 3.1262in;  depth 0pt;  original-width 3.2895in;
%original-height 3.2831in;  cropleft "0";  croptop "1";  cropright "1";
%cropbottom "0";  tempfilename 'T0S57V01.bmp';tempfile-properties "XPR";}} }%
%BeginExpansion
\begin{figure}[ptb]%
\centering
\includegraphics[
%natheight=3.283100in,barrier
%natwidth=3.289500in,
%height=2in,
width=0.95\textwidth
]%
{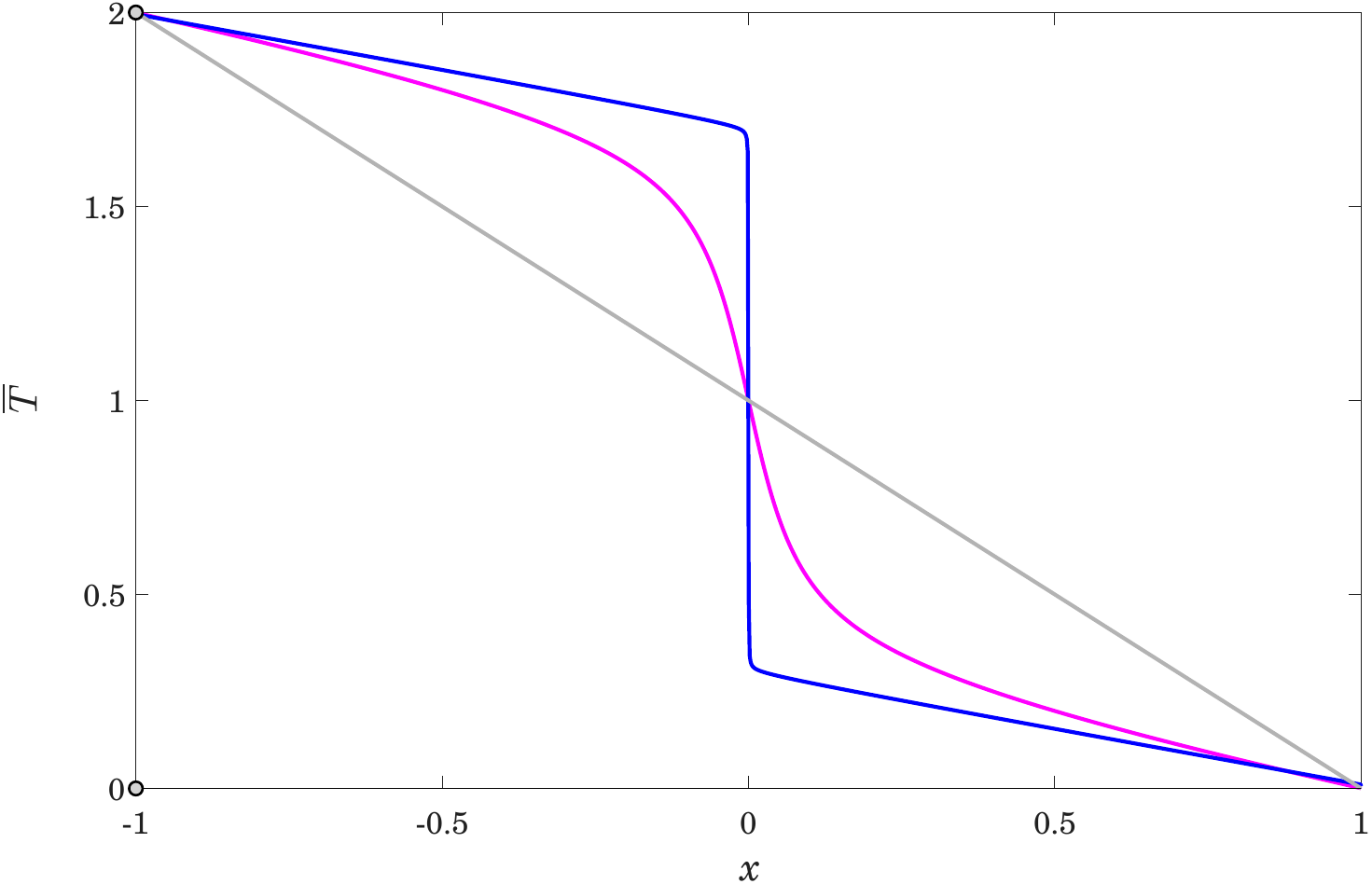}%
\caption{Temperature profiles $\overline{T}\left(  x\right)  $ in three cases:
$\epsilon=0.2$ (magenta), $\epsilon=0.01$ (blue), both based on equation
(\ref{example}), and the profile without barrier ($\overline\chi_\epsilon\equiv 1$, grey).}%
\end{figure}
%EndExpansion

As illustrated in Figure 1, we want to rescale the problem in such a way to
get a sharp profile (the blue one) corresponding to the soft one (the magenta
one) found in the example. Rigorously and in larger generality, this is done
in Theorem 10 below. Here we examine a simple case and argue only
heuristically: the case when, on a thin layer $x\in\left[  -\epsilon
,\epsilon\right]  $ the function $\overline{\chi}_{\epsilon}^{2}\left(
x\right)  $ behaves like $\left(  \frac{x}{\epsilon}\right)  ^{2}$.

We apply the general principle of \textit{equal heat fluxes between
compartments}, at equilibrium, principle that is easy understood by comparison
with queuing networks, or more rigorously from the fact seen above that the
heat flux $\left(  \kappa_{\epsilon}+\kappa_{T}^{\epsilon}\overline{\chi
}_{\epsilon}^{2}\left(  x\right)  \right)  \overline{T}_{\epsilon}^{\prime
}\left(  x\right)  $ is constant (the constant $\phi_{q}^{\epsilon}$ above).
We call $\phi_{q}^{\epsilon}$ this heat flux. We assume $\kappa_{\epsilon
}\ll\kappa_{T}^{\epsilon}$. On the interval $\left[  -1,-\epsilon\right]  $ the
heat flux $\phi_{q}^{\epsilon}$ is approximately equal to $\kappa
_{T}^{\epsilon}\left(  T_{+}-T_{\ast}\right)  $ where $T_{\ast}$ is the
temperature just on the left of $x=0$. Similarly on the interval $\left[
\epsilon,1\right]  $. But on the interval $\left[  -\epsilon,\epsilon\right]
$ the flux $\phi_{q}^{\epsilon}$ is approximately
\[
\phi_{q}^{\epsilon}\sim\frac{2T_{\ast}-T_{+}}{\int_{-\epsilon}^{\epsilon}%
\frac{1}{\kappa_{\epsilon}+\kappa_{T}^{\epsilon}\left(  \frac{x}{\epsilon
}\right)  ^{2}}dx}.
\]
Hence%
\[
\kappa_{T}^{\epsilon}\left(  T_{+}-T_{\ast}\right)  \sim\frac{2T_{\ast}-T_{+}%
}{\int_{-\epsilon}^{\epsilon}\frac{1}{\kappa_{\epsilon}+\kappa_{T}^{\epsilon
}\left(  \frac{x}{\epsilon}\right)  ^{2}}dx}.
\]
An easy calculation gives us%
\[
\Theta_{\epsilon}:=\int_{-\epsilon}^{\epsilon}\frac{\kappa_{T}^{\epsilon}%
}{\kappa_{\epsilon}+\kappa_{T}^{\epsilon}\left(  \frac{x}{\epsilon}\right)
^{2}}dx=\epsilon\frac{\sqrt{\kappa_{T}^{\epsilon}}}{\sqrt{\kappa_{\epsilon}}%
}2\arctan\left(  \sqrt{\frac{\kappa_{T}^{\epsilon}}{\kappa_{\epsilon}}%
}\right)  \sim\pi\epsilon\frac{\sqrt{\kappa_{T}^{\epsilon}}}{\sqrt
{\kappa_{\epsilon}}}%
\]
if $\kappa_{\epsilon}\ll\kappa_{T}^{\epsilon}$, and
\[
T_{\ast}\sim\frac{1+\Theta_{\epsilon}}{2+\Theta_{\epsilon}}T_{+}.
\]
We therefore see that in the case (with a free parameter $K>0$)%
\begin{equation}
\kappa_{\epsilon}=(K\epsilon)^{2}\kappa_{T}^{\epsilon}%
\label{scaling law}%
\end{equation}
we have a transport barrier as above, a ``partial'' barrier we could say, with
$\Theta_{\epsilon}=\frac\pi K$ and
\[
\frac{1}{2}<T_{\ast}\sim\frac{K+\pi }{2K+\pi }T_{+}<T_{+}.
\]
On the contrary, if $\frac{\kappa_{\epsilon}}{\epsilon^{2}\kappa_{T}%
^{\epsilon}}\rightarrow0$, we expect a full barrier: $T_{\ast}=T_{+}$, no heat
transport. While, if $\frac{\kappa_{\epsilon}}{\epsilon^{2}\kappa
_{T}^{\epsilon}}\rightarrow\infty$, there is no barrier, $T_{\ast}=\frac{1}%
{2}T_{+}$.

The scaling law is numerically confirmed: in the case of example
(\ref{example}) if we take%
\[
\epsilon=0.01,\qquad\kappa_{\epsilon}=\epsilon^{2}\kappa_{T}^{\epsilon},%
\]
we get the blue profile of Figure 2.

We therefore have a simple formula for the barrier height as a function of the
basic model parameters $\kappa_{\epsilon}$ and $\kappa_{T}^{\epsilon}$.
Moreover, we have a formula for the heat flux of the global system, which
indicates how much energy is ``saved'' by the barrier. The heat flux in the case
of turbulence without a barrier is equal to $\kappa_T^\epsilon\frac{T_{+}%
}{2}$. With a sharp barrier, it is equal to
\[
\phi_{q}^{\epsilon}\sim\kappa_{T}^{\epsilon}\left(  T_{+}-T_{\ast}\right)
\sim\kappa_{T}^{\epsilon}\frac{T_{+}K}{2K+\pi}<\kappa_{T}^{\epsilon}%
\frac{T_{+}}{2}.
\]
The barrier reduces heat transport, and the effect becomes stronger as 
$K$ decreases.

\section{Stochastic interpretation and the limit process\label{Section hard membrane}}\label{section:stoch interpret}

{Let us stress from the very beginning that we extend the results of this
section also to special cases where the function $\overline{\chi}_{\epsilon
}^{2}$ is not smooth, with a specific H\"{o}lder singularity. We think that
this illuminates more clearly the conditions for convergence as $\epsilon
\rightarrow0$.}

Consider the following SDE 

\begin{eqnarray}
\label{SDE comp 1}
dX_{\epsilon}\left(  t\right)  &=&2\kappa_{T}^{\epsilon}\overline{\chi
}_{\epsilon}\left(  X_{\epsilon}\left(  t\right)  \right)  \overline{\chi
}_{\epsilon}^{\prime}\left(  X_{\epsilon}\left(  t\right)  \right)  
dt+\sqrt{2\left(\kappa_{\epsilon}+\kappa_{T}^{\epsilon}\overline{\chi}_{\epsilon}^{2}\left(
X_{\epsilon}\left(  t\right)  \right) \right)}dw_{1}\left(  t\right), \ \ \\
dY_{\epsilon}\left(  t\right) &=& \sqrt{2\left(\kappa_{\epsilon}+\kappa_{T}^{\epsilon}\overline{\chi}_{\epsilon}^{2}\left(
X_{\epsilon}\left(  t\right)  \right) \right)}dw_{2}\left(  t\right),  \label{SDE comp 2}%
\end{eqnarray}
where $\left(  w_{1}\left(  t\right)  ,w_{2}\left(  t\right)  \right)
_{t\geq0}$ is a two-dimensional Wiener process. We consider 
$Y_{\epsilon}(t)$  and $Y_{\epsilon}(t)+2\pi n$ to be equivalent (or identified) for all 
$n\in\mathbb Z$.

It is well known that there is a close relationship between the PDE \eqref{eq:theor_1} and the SDE \eqref{SDE comp 1}.

We assume that the initial condition $\theta_0\in L^2(D)$ is smooth enough.
Since the function  $\overline{\chi}_{\epsilon}$ is smooth, there exists a unique solution to the problem \eqref{eq:theor_1}--\eqref{eq:theor_1_init}, which is smooth (e.g.  \cite{Ladyzhenskaya et al}).
 Then the solution to \eqref{eq:theor_1}--\eqref{eq:theor_1_init} can be expressed as  follows:
\begin{equation} \label{eq:sol_expect}
T_\epsilon(x,y,t)=\E_{x,y} \left[\theta_0(X_\epsilon(t), Y_\epsilon(t))\1_{\tau_\epsilon>t}\right]+ T_+\Pb_{x,y}(\tau_\epsilon\leq t, \ X_\epsilon(\tau_\epsilon)=-1)),
\end{equation}
where   $\tau_\epsilon:=\inf\{t\geq 0\colon |X_\epsilon(t)|=1\}.$
The stopping time $\tau_\epsilon$ is finite a.s. because the diffusion coefficient is non-degenerate. It is easy to see that    
\begin{equation}\label{eq:lim_T}
T_\epsilon(x,y):=\lim_{t\to\infty}T_\epsilon(x,y,t)= T_+\Pb_{x,y}( X_\epsilon(\tau_\epsilon)=-1).
\end{equation}
Moreover, it is well known that $T_{\epsilon}$ satisfies a stationary equation
\begin{equation}
\kappa_{\epsilon}\Delta T_{\epsilon}(x,y)+\kappa_{T}^{\epsilon}\operatorname{div}%
\left(  \overline{\chi}_{\epsilon}^{2}\left(  x\right)  \nabla T_{\epsilon
}(x,y)\right) =0. \label{stationary PDE}%
\end{equation}

\begin{remark}
    Formulas \eqref{eq:sol_expect}, \eqref{eq:lim_T}  remain valid if the function $\overline\chi_\epsilon$ is H\"older continuous.
\end{remark}

Note that the equation for 
$X_\epsilon$  does not depend on 
$Y_\epsilon$. Hence, to study the limiting behavior of 
$T_\epsilon$, it suffices to investigate the equation for $X_\epsilon$ alone. It is convenient to perform a  time transformation (see, for example, \cite{IkedaWatanabe1981}, Ch. III). Consider a continuous additive functional
$$
\beta_t:=\beta_t^\epsilon:=2\int_0^t \left(\kappa_{\epsilon}+\kappa_{T}^{\epsilon}\overline{\chi}_{\epsilon}^{2}\left(
X_{\epsilon}\left(  s\right)  \right) \right)ds.
$$
Since 
$\beta_t$
is a continuous, strictly increasing function, there exists an inverse function
$$
\alpha_t:=\alpha_t^\epsilon:=\inf\{s: \beta_s>t\}.
$$
We introduce the process $\left(\tilde X_\epsilon(t), \tilde Y_\epsilon(t)\right)_{t\geq0}:=\left( X_\epsilon(\alpha_t), Y_\epsilon(\alpha_t)\right)_{t\geq0}$. It can be shown (e.g. \cite{IkedaWatanabe1981}) that $\left(\tilde X_\epsilon(t), \tilde Y_\epsilon(t)\right)_{t\geq0}$  satisfies the SDE 
\begin{eqnarray}\label{eq:prelimit hard 1}
d\tilde X_{\epsilon}\left(  t\right)  &=&\frac{\kappa_{T}^{\epsilon}\overline{\chi
}_{\epsilon}\left(  \tilde X_{\epsilon}\left(  t\right)  \right)  \overline{\chi
}_{\epsilon}^{\prime}\left(  \tilde X_{\epsilon}\left(  t\right)  \right)  }%
{\kappa_{\epsilon}+\kappa_{T}^{\epsilon}\overline{\chi}_{\epsilon}^{2}\left(
\tilde X_{\epsilon}\left(  t\right)  \right)  }dt+d\tilde w_{1}\left(  t\right),
\\ \label{eq:prelimit hard 2}
d\tilde Y_{\epsilon}(t) &=& d\tilde w_{2}(t), %
\end{eqnarray}
where $\tilde w_i(t)=\int_0^{\alpha_t}\sqrt{2\left(\kappa_{\epsilon}+\kappa_{T}^{\epsilon}\overline{\chi}_{\epsilon}^{2}\left(
X_{\epsilon}\left(  s\right)  \right) \right)}dw_{i}(s)$, $i=1,2$, is a new Wiener process.
We see  that the process $(\tilde X_\epsilon(t))_{t\geq 0}$ behaves like a Brownian motion everywhere except in the neighborhoods $[-2\epsilon, 2\epsilon]$,  $[-1,-1+2\epsilon]$ and $[1-2\epsilon,1]$, where it receives local perturbations. 

Usually parameters $\kappa_\epsilon$ and $\kappa^\epsilon_T$ are small. We  assume that $\kappa^\epsilon_T/\kappa_\epsilon$ converges to $\infty$ as $\epsilon\to0.$ We will show the limit dynamics for $\tilde X_\epsilon$ exists and is non-trivial as $\epsilon\to0$, unlike  the limit of $X_\epsilon,$ which is a constant. 
The probabilities of hitting $-1$ before $1$ are the same for processes $X_\epsilon$ and $\tilde X_\epsilon$.   However, the investigation of the process  $\tilde X_\epsilon$ gives clear insight into the dynamics of plasma.

The main result of this section is the following. Under certain assumptions on the coefficients, we prove that the sequence $\tilde X_\epsilon$ converges to a diffusion in $[-1,1]$ with a hard membrane located at 0. This membrane does not allow a particle to penetrate immediately after the first hitting, but only after a certain number of visits. We introduce this process rigorously in Subsection \ref{sec:hard membrane}.    Moreover, we show that the limit of the corresponding probabilities $\Pb_{x,y}( \tilde X_\epsilon(\tilde \tau_\epsilon)=1)  $ as $\epsilon\to0$ coincides with the corresponding probability for Brownian motion with a hard membrane, which is explicitly calculated in the general case in Subsection \ref{sec:HM hitting probab}.

\subsection{A hard membrane}\label{sec:hard membrane}
In this section, we collect some facts about Brownian motion with a hard membrane (snapping out Brownian motion). Further information can be found in \cite{Lejay, ManPil}.

To introduce this process, let us first recall that  Brownian motion reflected at 0  into the positive   half-line is a     process   satisfying the following SDE
\begin{equation}\label{eq:reflection}
     w^+(t) = w(t) +  l^+(t),
\end{equation}
     where $w$ is a Brownian motion and  $l^+(t)=\lim_{\epsilon\to0+}\frac{1}{   2\varepsilon}\int_0^t\1_{|w^+(s)|\leq \varepsilon} ds$ is the symmetric local time of $w^+$ at 0. Recall that the distributions of the modulus of Brownian motion and of  reflected Brownian motion coincide. Existence, uniqueness, and other properties of reflected processes  can be found, for example, in \cite{Pilipenko2014} . The solution to \eqref{eq:reflection} is a strong Markov process whose increments coincide with those of $w$ when $w^+$ belongs to $(0,\infty)$. If the initial value  $w^+(0)$ is non-negative, then $w^+$ is a non-negative process and the pair $(w^+,l^+)$ solves the Skorokhod reflection problem. Moreover, the following formula holds: 
     $$
     l^+(t)=-\min_{s\in[0,t]}w(s)\wedge 0.
     $$

     Similarly, the equation for Brownian motion reflected into the negative half-line is
     \[
      d w^-(t) = dw(t) - d l^-(t),
     \]
where $l^-(t)=\lim_{\epsilon\to0+}\frac{1}{   2\varepsilon}\int_0^t\1_{|w^-(s)|\leq \varepsilon} ds$.

Let $X$ be a stochastic process that behaves like a  Brownian motion reflected  into  either the positive  or negative half-line, until its local time at $0$ reaches an exponentially distributed random variable. At that moment, it is killed and immediately reborn as a Brownian motion reflected into the opposite half-line, independently of the past. When the local time at 0 reaches another (independent) exponentially distributed random variable, the process is once more  killed and immediately  reborn as a Brownian motion reflected into  the opposite half-line, and so on. {In other words, $X$ behaves like an elastic Brownian motion on either the positive or negative side. At the moment of killing, it is immediately reborn at 0 and evolves, independently of the past, as an elastic Brownian motion on the opposite side of the real line until the next killing time, and so on. For the definition and further properties of elastic Brownian motion, see  \cite{ItoMcKean1965}.}

Let us formalize the construction.
Consider a continuous-time Markov process $N(t), t\geq 0,$ taking values in $\{-1,1\}$, which  is independent of a standard Wiener process  $(w(t), \mathcal{F}_t, t\geq 0)$, $w(0)=0$. Denote the transition intensities of $N$  by $\beta^+:=\lambda_{1,-1}, \beta^-:=-\lambda_{-1,1}$. It is well known that times between transitions from $1$ to $-1$ and from $-1$ to $1$ are independent exponentially distributed random variables: $(\xi_k^+)\sim \operatorname{Exp}(\beta^+), (\xi_k^-)\sim \operatorname{Exp}(\beta^-)$. 

Assume also that $\sigma(N(t), t\geq 0)\subset \mathcal{F}_0,$ and that the random variable $X(0)$ is  such that  $N(0)=1$  if $X(0)>0$   and $N(0)=-1$   if $X(0)<0$. Moreover, assume that $X(0)$ is independent of both the future jumps of $N$ and the Wiener process $w.$

We say that a continuous stochastic process $(X(t), \mathcal{F}_t, t\geq 0)$ is {\it a Brownian motion with a hard membrane} located at 0 and permeability parameters $\beta^\pm$, if $X$ is a solution to the SDE
\begin{equation}
    \label{eq:hard}
    X(t)=X(0)+w(t)+\int_0^t N(l(s)) d l(s),
\end{equation}
where $l(t):=\lim_{\varepsilon\to0+}(2\varepsilon)^{-1}\int_0^t\1_{|X(s)|\leq \varepsilon} ds$ is a local time of $X$ at 0.

It follows from the uniqueness of the solution to the equation for reflected Brownian motion that there exists a unique solution to \eqref{eq:hard}. Moreover, for all $t\geq 0$, it holds that 
\[
N(l(t)) \sgn(X(t))\geq 0.
\] 
It can be shown that the pair $(X(t), N(l(t))), t\geq 0,$ is a strong Markov process, whereas $X$ alone is, in general, is a Markov process, which does not possess the strong Markov property. The failure of the strong Markov property stems from the  fact that if $X(t)=0$, then in order to determine the future evolution of the process, one needs to know from which side of the origin the process arrived at 0 at time 
$t$.

\begin{remark} \label{rem:SNOB2lines}
    Instead of considering the strong Markov pair $(X(t), N(l(t))), t\geq 0,$ we can use an equivalent construction.  Recall that if $X(t)\neq 0,$ then $ N(l(t))=\sgn(X(t))$ and the pair $(X(t), N(l(t))$ is uniquely determined by its first coordinate. 
    We set 
    \[\tilde X(t):=\begin{cases}
        X(t),& \text{ if } X(t)\neq0,\\
        +0 ,& \text{ if } X(t)=0 \text{ and } N(X(t))=1,\\
        -0,& \text{ if } X(t)=0 \text{ and } N(X(t))=-1.
    \end{cases}
    \]
    Then $\tilde X$ is a strong Markov process on the disjoint union of two half-lines $(-\infty,-0]\cup[0+,\infty),$ where $\pm0$ are treated as two distinct points. Note also that
    \[X(t) =\begin{cases}
        \tilde X(t),& \text{if } \tilde X(t)\neq0,\\
        0 ,& \text{if } \tilde  X(t)=0 \pm.
    \end{cases}
    \]
    \end{remark}

\begin{remark}
    We introduce the filtration $(\mathcal{F}_t)$ primarily to ensure that standard results from the theory of SDEs can be applied {and to make the definition more concise}. However, the process 
$X$ can also be constructed explicitly by solving the Skorokhod reflection problem (see \cite{ManPil}) {and switching at the moments when the increments of its local time reach certain exponentially distributed random variables.}  %In this case,  the natural filtration generated by $X$ itself may be used instead. It can also be shown that the process $w$ is  a Brownian motion with respect to the filtration $(\mathcal{F}_t)$   as well.
\end{remark}

The following result shows that a Brownian motion process with a hard membrane can be obtained as the limit of Brownian motions that are  locally perturbed in a neighborhood of 0. 
Consider a sequence of  SDEs
\begin{eqnarray*}\label{eq_main}
dX_\epsilon(t)&=&a_\epsilon(X_{\epsilon}(t))dt +dw(t),\ t> 0,\\
X_{\epsilon}(0)&=&x,
\end{eqnarray*}
where $x\neq 0,$ $a_\epsilon$ is a continuous function with compact support $ [-\epsilon,\epsilon] $.  
Assume also that $\sgn(x) a_\epsilon(x)\geq 0$, i.e., the drift pushes $X_\varepsilon$ away from 0.

{Set $\sigma_{x}^{(\epsilon)}:=\inf\{t\geq 0\ : \ X_\epsilon(t)=x\}$.
\begin{theorem} \label{th: Pilip1}
\begin{enumerate}
    \item 
  Assume that $x_\epsilon\in[-\epsilon,\epsilon]$ are such that there exists the limit
\[
\lim_{\epsilon\to0}\epsilon^{-1}\Pb\Big\{\sigma_{x_\epsilon}^{(\epsilon)}< \sigma_{2\epsilon}^{(\epsilon)} \  | \ X_\epsilon(0)=\epsilon \Big\} =\beta^+\in(0,\infty).
\]
If $ X_\epsilon(0)=x>0,$ then
\[\Big(  X_\epsilon(\cdot\wedge \sigma_{x_\epsilon}^{(\epsilon)}), \sigma_{x_\epsilon}^{(\epsilon)}\Big) \Rightarrow \Big(|w|(\cdot\wedge \sigma_{+}),\sigma_+\Big), \epsilon\to0,\]
     where $w$ is a Brownian motion started from x, $\sigma_+=\inf\{t\geq 0 \colon l(t)\geq \operatorname{Exp}(\beta^+)\},$ $l$ is the symmetric local time of $w$ at 0, $\operatorname{Exp}(\beta^+)$ is an exponential random variable with parameter $\beta^+$, independent of $w.$
\item  Assume that $x_\epsilon\in[-\epsilon,\epsilon]$ are such that there exists the limit
\[
\lim_{\epsilon\to0}\epsilon^{-1}\Pb\Big\{\sigma_{x_\epsilon}^{(\epsilon)}< \sigma_{-2\epsilon}^{(\epsilon)} \  | \ X_\epsilon(0)=-\epsilon \Big\} =\beta^-\in(0,\infty).
\]
If $ X_\epsilon(0)=x<0,$ then
\[\Big(  X_\epsilon(\cdot\wedge \sigma_{x_\epsilon}^{(\epsilon)}), \sigma_{x_\epsilon}^{(\epsilon)}\Big) \Rightarrow \Big(-|w|(\cdot\wedge \sigma_{-}),\sigma_-\Big), \epsilon\to0,\]
     where $w$ is a Brownian motion started from x, $\sigma_-=\inf\{t\geq 0 \colon l(t)\geq \operatorname{Exp}(\beta^-)\},$ $l$ is the symmetric local time of $w$ at 0, and $\operatorname{Exp}(\beta^-)$ is an exponential random variable with parameter $\beta^-$, independent of $w.$
\item 
     If $ X_\epsilon(0)=x\neq 0,$ and there exist the limits 
\[
\lim_{\epsilon\to0}\epsilon^{-1}\Pb\Big\{\sigma_{\mp\epsilon}^{(\epsilon)}< \sigma_{\pm2\epsilon}^{(\epsilon)} \  | \ X_\epsilon(0)=\pm\epsilon \Big\} =\beta^\pm\in(0,\infty),
\]
then    $X_\epsilon$ converges in distribution to a Brownian motion with a hard membrane, started from $0$, with parameters $\beta^\pm.$
 \end{enumerate}   
 \end{theorem}
 The proof follows from Theorem 1 and the arguments of Theorem  2  in \cite{ManPil}.  
In this theorem, the case $x_\epsilon=\pm \epsilon$ was considered,    but  the proof remains valid for $x_\epsilon$ satisfying the assumptions above. 
}

Set
\begin{eqnarray*}
  p_\epsilon^+ &:=& \Pb\Big\{\inf\{t\geq 0:X_\epsilon(t)=-\epsilon\}<\inf\{t\geq 0:\ X_\epsilon(t)=2\epsilon\} \  | \ X_\epsilon(0)=\epsilon \Big\}, \\
   p_\epsilon^- &:=& \Pb\Big\{\inf\{t\geq 0:X_\epsilon(t)=\epsilon\}<\inf\{t\geq 0:\ X_\epsilon(t)=-2\epsilon\} \  | \ X_\epsilon(0)=-\epsilon \Big\}.
\end{eqnarray*}
\begin{theorem} \label{th: Pilip}
Assume that
$$ 
\lim_{\epsilon\to 0+}\epsilon^{-1}   p_\epsilon^\pm= \beta^\pm\in(0,\infty).
$$
Then  $X_\epsilon$ converges in distribution to a Brownian motion with a hard membrane and parameters $\beta^\pm.$
\end{theorem}

This theorem will be used in Subsection \ref{sec:limit hard}, where we prove that the process $\tilde X_\epsilon$ converges to a Brownian motion with a hard membrane as $\epsilon\to0.$

To calculate hitting probabilities for a Brownian motion with a hard membrane, we need another approximating result. Let us construct a process $\tilde X_\epsilon$,  $\epsilon>0$ as follows. Assume that $\tilde X_\epsilon(0)>0 $ (or $\tilde X_\epsilon(0)<0$).   Then $\tilde X_\epsilon$ behaves as a Brownian motion until the first hitting 0. At this instant the process jumps to $-\epsilon $ with probability $p_{\epsilon}^+$ and jumps to $\epsilon$ with probability $1-p_{\epsilon}^+ $ independently of the past (and correspondingly jumps to $\epsilon $ with probability $p_{\epsilon}^-$ and to $-\epsilon$ with probability $1-p_{\epsilon}^-$ if approaching from below). After this, the process $\tilde X_\epsilon$ continues to follow the described algorithm independently of the past.
Denote by $\zeta_\epsilon$ the first instant, when $\tilde X_\epsilon$ changes the sign.

\begin{theorem}\label{th:conv_bar_X}
    Let $ \tilde X_\epsilon(0)=x\neq 0 $. Suppose there exist  limits $\beta^\pm:=\lim_{\epsilon\to0}\epsilon^{-1}p_\pm^\epsilon\in(0,\infty).$ 
    \begin{enumerate}
        \item 
    If $ x>0$, then the stopped processes converge in distribution
     \[(\tilde X_\epsilon(\cdot\wedge \zeta_\epsilon), \zeta_\epsilon) \Rightarrow (|w|(\cdot\wedge \zeta_+),\zeta_+), \epsilon\to0,\]
     where $w$ is a Brownian motion started from 0, $\zeta_+=\inf\{t\geq 0 \colon l(t)\geq \operatorname{Exp}(\beta^+)\},$ $l$ is the symmetric local time of $w$ at 0, $\operatorname{Exp}(\beta^+)$ is an exponential random variable with parameter $\beta^+$, independent of $w.$
     \item If $ x<0$, then  
     \[(\tilde X_\epsilon(\cdot\wedge \zeta_\epsilon), \zeta_\epsilon) \Rightarrow (-|w|(\cdot\wedge \zeta_-),\zeta_-), \epsilon\to0,\]
     where $\zeta_-=\inf\{t\geq 0 \colon l(t)\geq \operatorname{Exp}(\beta^-)\}.$
     \item 
 $\tilde X_\epsilon$ converges in distribution to a Brownian motion with a hard membrane with parameters $\beta^\pm$, starting from $x$.
    \end{enumerate} 
        The same conclusions hold for the process $  X_\epsilon$ in place of $\tilde X_\epsilon.$
\end{theorem}
The first two statements follow from Theorem 1 and the reasoning in Theorem  2  of \cite{ManPil}; the third is a consequence of the first two.  
\begin{remark}
 If $\beta^+=0 $  in the previous theorem,  the result still holds with the following modification. If $x>0$, then the limit process exists and coincides with a  Brownian motion reflected at 0 into the positive half-line.  In this case, we say that the membrane at 0  is non-permeable from the right. If  $x<0$, then $X$ behaves as (negative) elastic Brownian motion with parameter $\beta^-$; upon killing, it switches to the opposite half-line and continues as a reflected Brownian motion.
\end{remark}
\begin{remark}\label{rem:limitSkew}
The case $\beta_\pm=\infty$ is critical and requires additional considerations. Assume that the limits $\lim_{\epsilon\to0}p^\epsilon_\pm=:p_\pm\in(0,1)$  exist. Then, similarly to Theorem 3 in \cite{PilipenkoPrykhodko2015} it can be shown that $X_\epsilon\Rightarrow X$, as $\epsilon \to 0$, where $X$ is a skew Brownian motion with permeability parameter $\gamma=\frac{p_--p_+}{p_-+p_+}\in(-1,1)$. For the definition and other  properties of a skew Brownian motion see \cite{BarlowPitmanYor, Iksanov2025, Lejay2006}.
\end{remark}

\subsection{Hitting probabilities}\label{sec:HM hitting probab}
In this subsection, we calculate the hitting probabilities for a Brownian motion with a hard membrane.

In the next theorem, it is more convenient to consider a strong Markov `modification' $\tilde X$ defined in Remark \ref{rem:SNOB2lines}. We will also refer to it as a Brownian motion with a hard membrane.
\begin{theorem} \label{thm:probab_hard}
    Let $\tilde X$ be a Brownian motion with a hard membrane at 0 with permeability parameters $\beta^\pm$.
    Then for any $a<0<b$ and any $x\in[a,b]\setminus \{0\}$, we have
    \begin{equation}\label{eq:P_tau_a_b}
    \Pb_x(\tau_a<\tau_b)=\begin{cases}
        \frac{x}{a}+\frac{a-x}{a} \Pb_{0-}(\tau_a<\tau_b),& x\in[a,0),\\
         \frac{b-x}{b} \Pb_{0+}(\tau_a<\tau_b),& x\in(0,b],
    \end{cases}
  \end{equation}
  \begin{equation}\label{eq:P_tau_b_a}
    \Pb_x(\tau_a>\tau_b)=\begin{cases}
         \frac{a-x}{a} \Pb_{0-}(\tau_a>\tau_b),& x\in[a,0),\\
       \frac{x}{b}+  \frac{b-x}{b} \Pb_{0+}(\tau_a>\tau_b),& x\in(0,b],
    \end{cases}
    \end{equation}
    where 
    \[
    \tau_y:=\inf\{t\geq 0\colon \tilde X(t)=y\},
    \]
    and
    \begin{eqnarray}
    \Pb_{0-}(\tau_a<\tau_b)&=&1-\Pb_{0-}(\tau_a>\tau_b)=1-\frac{-a\beta^-}{-a\beta^-+b\beta^++1}=\frac{b\beta^++1}{-a\beta^-+b\beta^++1}, \nonumber\\
    \Pb_{0+}(\tau_a<\tau_b)&=&1-\Pb_{0+}(\tau_a>\tau_b)= 1-\frac{-a\beta^- +1}{-a\beta^-+b\beta^++1}=
    \frac{b\beta^+}{-a\beta^-+b\beta^++1}.\ \quad \quad \label{eq:tau_1_tau_b}
    \end{eqnarray}
\end{theorem}
\begin{proof}
Let $w$ be a Brownian motion and $l$ its symmetric local time at 0. Define $\tau^{|w|}_b:=\inf\{t\geq 0\colon |w(t)|=b\})$,  $b>0$. It is known, see \cite[formula (4.19) at p. 429]{KaratzasShreve1991} or  \cite[formula (3.18.1)]{BorodinSalminen}, that if $w(0)=0,$ then the accumulated local time $l(\tau^{|w|}_b)$  has an exponential distribution $\operatorname{Exp}(b^{-1})$ with mean $b$.
Let $Y$ be a (positive) elastic Brownian motion with parameter $\beta>0$ defined as \[
Y(t)=\begin{cases}
    |w(t)|, & l(t)<\operatorname{Exp}(\beta),\\
    \infty, & l(t)\geq \operatorname{Exp}(\beta),
\end{cases}\]
where  $\operatorname{Exp}(\beta)$ is an exponential random variable with parameter $\beta$, independent of $w$. Then the probability that $Y$ hits $b$, given that $Y(0)=0$, is equal to
\begin{align}
    \Pb_0(\exists t\geq 0 \colon Y(t)=b)&=\Pb_0(l(\tau^{|w|}_b)<\operatorname{Exp}(\beta)) \nonumber
    \\  =&\Pb(\operatorname{Exp}(b^{-1})<\operatorname{Exp}(\beta))=\frac{b^{-1}}{b^{-1}+\beta}=\frac{1}{1+b\beta}. \label{eq:probab_elastic}
\end{align}
Recall that the Brownian motion with a hard membrane behaves as a standard Brownian motion up to hitting 0. Thus,
for any $x_1x_2>0$ and $x\in[x_1,x_2]$,    the probability
\begin{equation}\label{eq:hit_prob_BM}
\Pb_x(\tau_{x_1}<\tau_{x_2})=\frac{x_2-x}{x_2-x_1}    
\end{equation}
coincides with the corresponding probability for standard Brownian motion. Applying the strong Markov property for $\tilde X$ and taking into account that the distributions of the modulus of Brownian motion and of reflected Brownian motion coincide,  we obtain \eqref{eq:P_tau_a_b}, \eqref{eq:P_tau_b_a}.
To calculate the probabilities $\Pb_{0\pm}(\tau_a<\tau_b)$, we apply the strong Markov property once again together with \eqref{eq:probab_elastic}:
\begin{align*}
    \Pb_{0-}(\tau_a<\tau_b)=\Pb_{0-}(\tau_a<\tau_{0+})+ \Pb_{0-}(\tau_a>\tau_{0+})\Pb_{0+}(\tau_a<\tau_b) \\
    =\frac{1}{1-a\beta_-}+\frac{-a\beta_-}{1-a\beta_-}\Pb_{0+}(\tau_a<\tau_b),
\end{align*}
\begin{align*}
    \Pb_{0+}(\tau_a<\tau_b)=  \Pb_{0+}(\tau_{0-}<\tau_b)\Pb_{0-}(\tau_a<\tau_b) \\
   = \frac{b\beta_+}{1+b\beta_+}\Pb_{0-}(\tau_a<\tau_b).
\end{align*}
Solving this system of linear equations yields the explicit expressions for $ \Pb_{0\pm}(\tau_a<\tau_b)$, and thus completes the proof of the theorem.
\end{proof}

\subsection{Convergence to a hard membrane process}\label{sec:limit hard}
In this section, we study the stationary problem given by equations  \eqref{stationary PDE},  \eqref{bound_cond_left}, and \eqref{bound_cond_right}, which is associated with the initial-boundary value problem \eqref{eq:theor_1}--\eqref{eq:theor_1_init}.

The central question is whether this process converges as $\epsilon\rightarrow0$, under suitable assumptions on the parameters $\kappa_{\epsilon},\kappa_{T}^{\epsilon}$, and the functions
$\chi_{\epsilon,ver}\left(  x\right)$, as well as%
\[
\overline{\chi}_{\epsilon}\left(  x\right)  =\chi_{\epsilon,bdr}\left(
x-1\right)  \chi_{\epsilon,bdr}\left(  x+1\right)  \chi_{\epsilon}\left(
x\right)  .
\]

Recall that the solution to problem (9) is represented in the form (see \eqref{eq:lim_T})
\[
T_\epsilon(x,y)= T_+\Pb_{x,y}( X_\epsilon(\tau_\epsilon)=-1),
\]
where $X_\epsilon$ is the solution to equation \eqref{SDE comp 1}, and $\tau_\epsilon:=\inf\{t\geq 0\colon |X_\epsilon(t)|=1\}.$ 

{This probability can be calculated explicitly.  We recall a well-known result from the theory of one-dimensional diffusions \cite{Gikhman_Skorokhod_1968}. Assume that $X$ is a solution to the SDE
\[
dX(t)=\alpha(X(t)) dt+\beta(X(t)) d w (t),
\]
where $\alpha, \sigma $ are Lipschitz continuous functions such that $\sigma(y)\neq 0, y\in\mbR.$ Let $\tau_y:=\inf\{t\geq 0\colon X(t)=y\}$ be the first hitting time of level $y$. Then, for any $x_1\leq x\leq x_2$,  
\[
\Pb_x(\tau_{x_1}<\tau_{x_2})=1-\Pb_x(\tau_{x_2}<\tau_{x_1})=\frac{s(x_2)-s(x)}{s(x_2)-s(x_1)},
\]
where $s$ is any non-constant solution of the differentioal equation \[\alpha(y)s'(y)+1/2 \beta^2(y)s''(y)=0, y\in\mbR. \] In particular, we can take 
\[
s(y):= \int_{y_0}^y\exp\left(-\int_{z_0}^z\frac{2\alpha(u)}{\beta^2(u)}du\right) dz, 
\]
for arbitrary  $y_0, z_0\in\mathbb{R}$. The function $s$ is called the scale function.}

%It follows from this relation that
Hence, in order to study the limit of stationary solutions as $\epsilon\to 0$, it is sufficient to analyze the scale functions of the processes $X_\epsilon$ alone. However, we will establish a more general result concerning the convergence of  processes $\{\tilde X_\epsilon\}$, obtained from the processes $\{X_\epsilon\}$ by a time change (see Section \ref{section:stoch interpret}), to a Brownian motion with  a hard membrane as $\epsilon \to 0$, under certain assumptions on the relations between $\kappa_{\epsilon}$ and $\kappa_T^{\epsilon}$. 

It is clear  that the exit probability $\Pb_{x,y}( X_\epsilon(\tau_\epsilon)=-1) $ coincides with that of the process $\tilde X_\epsilon$. 
From this, by applying Lemma \ref{lem:cont_tau}, we obtain the desired conclusion.

We first obtain the result for the process  $\bar X_\epsilon$, which is identical to $\tilde X_\epsilon$ except that it has no barriers at $\pm 1$, so that it can pass freely through the lateral boundaries.
\begin{theorem}\label{th:concergence without boundaties}
    Let $\bar X_\epsilon$ be a solution to the SDE
\begin{equation}\label{eq:conv without boundaries}
    d\bar X_{\epsilon}\left(  t\right)  =\frac{\kappa_{T}^{\epsilon} {\chi
}_{\epsilon}\left(  \bar X_{\epsilon}\left(  t\right)  \right)   {\chi
}_{\epsilon}^{\prime}\left(  \bar X_{\epsilon}\left(  t\right)  \right)  }%
{\kappa_{\epsilon}+\kappa_{T}^{\epsilon} {\chi}_{\epsilon}^{2}\left(
\bar X_{\epsilon}\left(  t\right)  \right)  }dt+d  w \left(  t\right),
\end{equation}
      where   $ {\chi}_{\epsilon}\left(  x\right)$ is an even function such that
   \begin{equation}
       \label{eq:chi_e}
   \chi_{\epsilon}\left(  x\right)  =\left\{
\begin{array}
[c]{ccc}%
\frac{1}{2}\left(  \frac{x}{\epsilon}\right)  ^{\alpha} & \text{for} &
x\in\left[  0,\epsilon\right), \\
-\frac{1}{2}\frac{\left\vert x-2\epsilon\right\vert ^{\alpha}}{\epsilon
^{\alpha}}+1 & \text{for} & x\in\left[  \epsilon,2\epsilon\right),  \\
1 & \text{for} & x\geq2\epsilon,
\end{array}
\right.
   \end{equation}
and $\alpha>0$ is a fixed parameter. Suppose that $\bar X_{\epsilon}\left(  0\right)\neq 0$.

Assume that $\frac{\kappa_T^\epsilon}{\kappa_\epsilon}\to \infty$ as $\epsilon \to 0$ and let $K$ be a constant.

\begin{enumerate}
    \item  If $\alpha>1/2$ and
 \begin{enumerate}
 \item 
    $\frac{\kappa_T^\epsilon}{\kappa_\epsilon}\sim (K\epsilon)^{{2\alpha}/{(1-2\alpha)}}$, as $\epsilon\to 0$, then $\bar X_\epsilon$ converges in distribution as $\epsilon \to 0$  to a Brownian motion with a hard membrane   
 and parameters $\beta^+=\beta^-  =K\left(\int_0^\infty\frac{dv}{1+\frac14 v^{2\alpha}}\right)^{-1} $;
 %and parameters $\beta^+=\beta^-=\beta$.
 \item  $\frac{\kappa_T^\epsilon}{\kappa_\epsilon}= o( \epsilon^{{2\alpha}/{(1-2\alpha)}})$, as $\epsilon\to 0$, then $\bar X_\epsilon$ converges in distribution as $\epsilon \to 0$ to a Brownian motion;
\item   $\epsilon^{{2\alpha}/{(1-2\alpha)}}= o\left(\frac{\kappa_T^\epsilon}{\kappa_\epsilon}\right) $, as $\epsilon\to 0$, then $\bar X_\epsilon$ converges in distribution as $\epsilon \to 0$  to a    Brownian motion reflected at 0 with nonpermeable membrane at 0.
\end{enumerate}
\item
 If  $\alpha=1/2$ and 
 \begin{enumerate}
     \item 
 $\ln\frac{\kappa_T^\epsilon}{\kappa_\epsilon}\sim (K\epsilon)^{-1}$, as $\epsilon\to 0$, then  $\bar X_\epsilon$ also converges  in distribution  as $\epsilon \to 0$ to a Brownian motion with a hard membrane and parameters $\beta^+=\beta^-  =\frac{K}{4}$;
  \item $\ln\frac{\kappa_T^\epsilon}{\kappa_\epsilon}=o(\epsilon^{-1})$, as $\epsilon\to 0$,  then $\bar X_\epsilon$ converges in distribution as $\epsilon \to 0$  to a Brownian motion;
\item  
$\epsilon^{-1}=o\left(\ln\frac{\kappa_T^\epsilon}{\kappa_\epsilon}\right)$, as $\epsilon\to 0$, then $\bar X_\epsilon$ converges  in distribution as $\epsilon \to 0$ to a    Brownian motion reflected at 0 with nonpermeable membrane at 0.
 \end{enumerate}
 \item
If {$\alpha\in(0,1/2)$}  %, and $\frac{\kappa_T^\epsilon}{\kappa_\epsilon}\to \infty$
 as $\epsilon \to 0$, then  $\bar X_\epsilon$ converges  in distribution  as $\epsilon \to 0$ to a Brownian motion.
\end{enumerate}

% there exists a limit 
% $$
% \lim_{\epsilon\to 0} p_\epsilon^{\pm}=\lim_{\epsilon \to 0}\left(\left(\frac{\kappa_T^\epsilon}{\kappa_\epsilon}\right)^{2\alpha-1}\int_0^{{\kappa_T^\epsilon}/{\kappa_\epsilon}}\frac{dv}{1+\frac14v^{2\alpha}}+3\right)^{-1}>0.
% $$ 
  \label{thm:barX_hardMembrane}
\end{theorem}
{\begin{remark}
The drift function $\frac{\kappa_{T}^{\epsilon} {\chi
}_{\epsilon}\left(  x  \right)   {\chi
}_{\epsilon}^{\prime}\left( x\right)  }%
{\kappa_{\epsilon}+\kappa_{T}^{\epsilon} {\chi}_{\epsilon}^{2}\left(
x \right)  }$ in equation \eqref{eq:conv without boundaries} is non-differentiable at the point 0. 
However, for $\alpha>0$ it is locally integrable at $0$.
Therefore, for 
$\alpha>0$ there exists a  unique  solution of equation (1) (see, e.g., \cite{ChernyEngelbert2005}).
\end{remark}
}
\begin{proof}[Proof of Theorem \ref{thm:barX_hardMembrane}]
The scale function of the process $\left(  \bar X_{\epsilon}\left(
t\right)  \right)  _{t\geq0}$ is given by
\[
s_{\epsilon}\left(  x\right)  =\int_{0}^{x}\exp\left\{  -2\int_{0}^{y}%
\frac{\kappa_{T}^{\epsilon}\overline{\chi}_{\epsilon}\left(  z\right)
\overline{\chi}_{\epsilon}^{\prime}\left(  z\right)  }{\kappa_{\epsilon
}+\kappa_{T}^{\epsilon}\overline{\chi}_{\epsilon}^{2}\left(  z\right)
}dz\right\}  dy.
\]
The exponent can be rewritten as:
\begin{align*}
-2\int_{0}^{y}\frac{\kappa_{T}^{\epsilon}{\chi}_{\epsilon}\left(
z\right) {\chi}_{\epsilon}^{\prime}\left(  z\right)  }%
{\kappa_{\epsilon}+\kappa_{T}^{\epsilon}{\chi}_{\epsilon}^{2}\left(
z\right)  }dz  & =-\int_{0}^{y}\frac{d\left(  \kappa_{\epsilon}+\kappa
_{T}^{\epsilon}{\chi}_{\epsilon}^{2}\left(  z\right)  \right)
}{\kappa_{\epsilon}+\kappa_{T}^{\epsilon}{\chi}_{\epsilon}^{2}\left(
z\right)  }dz=\left.  -\log\left(  \kappa_{\epsilon}+\kappa_{T}^{\epsilon
}{\chi}_{\epsilon}^{2}\left(  z\right)  \right)  \right\vert
_{z=0}^{z=y}\\
& =\log\frac{\kappa_{\epsilon}}{\kappa_{\epsilon}+\kappa_{T}^{\epsilon
}{\chi}_{\epsilon}^{2}\left(  y\right)  }.
\end{align*}
Thus,
\[
s_{\epsilon}\left(  x\right)  =\int_{0}^{x}\frac{\kappa_{\epsilon}}%
{\kappa_{\epsilon}+\kappa_{T}^{\epsilon}{\chi}_{\epsilon}^{2}\left(
y\right)  }dy.
\]

Let us compute%
\[
\bar p_{\epsilon}^{+}=\Pb\Big( \inf\big\{t\geq 0:\bar X_\epsilon(t)=0\}<\inf\{t\geq 0:\ \bar X_\epsilon(t)=4\epsilon\big\} \ \Big|\bar X_{\epsilon}\left(
0\right)  =2\epsilon\Big).
\]
It is given by%
\[
\bar p_{\epsilon}^{+}=\frac{s_{\epsilon}\left(  4\epsilon\right)  -s_{\epsilon
}\left(  2\epsilon\right)  }{s_{\epsilon}\left(  4\epsilon\right)
-s_{\epsilon}\left(  0\right)  }=\frac{\int_{2\epsilon}^{4\epsilon}}{\int%
_{0}^{\epsilon}+\int_{\epsilon}^{2\epsilon}+\int_{2\epsilon}^{4\epsilon}},%
\]
where inside the integrals we have the integrand $\frac{\kappa_{\epsilon}%
}{\kappa_{\epsilon}+\kappa_{T}^{\epsilon}{\chi}_{\epsilon}^{2}\left(
y\right)  }$. Put $\eta^{2\alpha}=\frac{\kappa_{T}^{\epsilon}}{\kappa
_{\epsilon}}$. Assume $\epsilon$ is small enough so that the boundary
layer and the transport barrier do not intersect, i.e. $-1+4\epsilon<-4\epsilon$.
Then%
\[
\int_{0}^{\epsilon}\frac{\kappa_{\epsilon}}{\kappa_{\epsilon}+\kappa
_{T}^{\epsilon}{\chi}_{\epsilon}^{2}\left(  y\right)  }dy=\int%
_{0}^{\epsilon}\frac{dy}{1+\frac{\kappa_{T}^{\epsilon}}{4\kappa_{\epsilon}%
}\left(  \frac{y}{\epsilon}\right)  ^{2\alpha}}=\epsilon\int_{0}^{1}\frac
{du}{1+\frac{\kappa_{T}^{\epsilon}}{4\kappa_{\epsilon}}u^{2\alpha}}%
=\frac{\epsilon}{\eta}\int_{0}^{\eta}\frac{dv}{1+\frac{1}{4}v^{2\alpha}}.
\]
There exists $c\in\left[  \epsilon,2\epsilon\right]  $ such that%
\[
\int_{\epsilon}^{2\epsilon}\frac{\kappa_{\epsilon}}{\kappa_{\epsilon}%
+\kappa_{T}^{\epsilon}{\chi}_{\epsilon}^{2}\left(  y\right)
}dy=\epsilon\frac{1}{1+\eta^{2\alpha}{\chi}_{\epsilon}^{2}\left(
c\right)  },
\]
where $\frac{1}{4}\leq{\chi}_{\epsilon}^{2}\left(  c\right)  \leq1$,
and%
\[
\int_{2\epsilon}^{4\epsilon}\frac{\kappa_{\epsilon}}{\kappa_{\epsilon}%
+\kappa_{T}^{\epsilon}{\chi}_{\epsilon}^{2}\left(  y\right)
}dy=2\epsilon\frac{1}{1+\eta^{2\alpha}}.
\]
Hence,%
\begin{equation}
\bar p_{\epsilon}^{+}\epsilon^{-1}={\epsilon}^{-1}\frac{\frac{1}{1+\eta^{2\alpha}}}{\frac{1}{\eta
}\int_{0}^{\eta}\frac{dv}{1+\frac{1}{4}v^{2\alpha}}+\frac{1}{1+\eta^{2\alpha
}{\chi}_{\epsilon}^{2}\left(  c\right)  }+\frac{2}{1+\eta^{2\alpha}}%
}.\label{key identity}%
\end{equation}

Let $\alpha>\frac{1}{2}$.
By the assumptions of the theorem, $\eta\rightarrow\infty$ as $\epsilon \to 0$. Then  
$$
 {\bar p^+_\epsilon}{\epsilon^{-1}}\sim \frac{1}{\epsilon\left(\frac{1}{\chi^2(c)}+2\right)+\epsilon\eta^{2\alpha-1}\int_0^\infty\frac{dv}{1+\frac14 v^{2\alpha}}}\sim \frac{1}{\epsilon\eta^{2\alpha-1}\int_0^\infty\frac{dv}{1+\frac14 v^{2\alpha}}}, \ \epsilon \to 0. %\eta\rightarrow\infty. 
$$
If % $\epsilon\sim \eta^{1-2\alpha}$,
$\eta^{2\alpha}=\frac{\kappa_T^\epsilon}{\kappa_\epsilon}\sim (K\epsilon)^{{2\alpha}/{(1-2\alpha)}}$, i. e. $\eta\sim (K\epsilon)^{1/(1-2\alpha)}$, there exists 
a limit
\[
\beta^+:=\lim_{\epsilon\to 0}{\bar p^+_\epsilon}{\epsilon^{-1}}=\left(\frac1K\int_0^\infty\frac{dv}{1+\frac14 v^{2\alpha}}\right)^{-1}.
\] 

Since $\chi_\epsilon$ is even, the same applies to
$$
\bar p^-_\epsilon =\Pb\Big( \inf\big\{t\geq 0:\bar X_\epsilon(t)=0\}<\inf\{t\geq 0:\ \bar X_\epsilon(t)=-4\epsilon\big\} \ \Big|\bar X_{\epsilon}\left(
0\right)  =-2\epsilon\Big),
$$
and
$\beta^-:=\lim_{\epsilon\to 0}\bar p^-_\epsilon \epsilon^{-1}=\beta^+$.
 By {Theorem \ref{th: Pilip1}}, $\bar X_\epsilon$ converges in distribution to a Brownian motion with a hard membrane with parameters $\beta^\pm=\beta$.
 
For $\alpha=1/2$, 
$$
 {\bar p^+_\epsilon}{\epsilon^{-1}}\sim \frac{1}{4\epsilon\frac{\eta}{1+\eta}\int_0^\eta\frac{dv}{4+v}} \sim\left(4\epsilon(\log(\eta+4)-\log 4)\right)^{-1}, \ \epsilon\to 0.
$$  
If  $\log\frac{\kappa_T^\epsilon}{\kappa_\epsilon}\sim (K\epsilon)^{-1}$, then there exists a limit
\[
\beta^+=\beta^-=\lim_{\epsilon\to 0}\bar p^+_\epsilon \epsilon^{-1}=\lim_{\epsilon\to 0} \left(4\epsilon\log\frac{\kappa_T^\epsilon}{\kappa_\epsilon}\right)^{-1}=\frac{K}{4},
\]
and $\bar X_\epsilon$ converges in distribution to a Brownian motion with a hard membrane with parameter $\beta^{\pm}=\beta$.

If $\alpha<1/2$, then 
 \[
\lim_{\epsilon \to 0}p_\epsilon^+=\lim_{\epsilon \to 0}\left(\frac{\int_0^\eta\frac{dv}{1+\frac14v^{2\alpha}}}{\eta^{2\alpha-1}}+3\right)^{-1}=\left(\frac{4}{1-2\alpha}+3\right)^{-1}=\frac{1-2\alpha}{7-6\alpha},
\]
which follows easily from l'H\^opital's rule. The rest of the proof follows from Remark \ref{rem:limitSkew}.
\end{proof}

\begin{corollary}
Let  $\bar\tau_{\epsilon}=\inf\{t\geq0:|\bar X_\epsilon (t)|=1\}$.    Then $\lim_{\epsilon\to 0 } \Pb_{x }(\bar X_\epsilon(\bar\tau_\epsilon)=-1)$ coincides with $\Pb_{x,y}( X(\tau)=-1)$, where $X$ is a Brownian motion with a hard membrane, $\tau=\inf\{t\geq 0: |X(t)|=1\}$. In particular,
\begin{equation} \label{eq:prob_1}
\lim_{\epsilon\to 0 } \Pb_{x }(\bar X_\epsilon(\bar\tau_\epsilon)=-1)=\begin{cases} \frac{1+(1-x)\beta}{2\beta+1}, & x\in [-1,0),\\
\frac{(1-x)\beta}{2\beta+1},& x\in (0,1].\end{cases}
\end{equation}
\end{corollary}
The proof of the ccorollary follows from Lemma \ref{lem:cont_tau}. Formula \eqref{eq:prob_1} is obtained by substituting $a=-1$, $b=1$, $\beta^+=\beta^-=\beta$ into \eqref{eq:P_tau_a_b}.

\begin{lemma}\label{lem:cont_tau}
    Assume that a sequence of c\`adl\`ag processes $(Z_\epsilon)$ converges in distribution to a process $Z_0$ in the space $D([0,\infty))$ of c\` adl\`ag functions with Skorokhod's topology, where $Z_0$ is a continuous process. Let 
    \[
    \lim_{\epsilon\to0} a_\epsilon=a_0<b_0=\lim_{\epsilon\to0} b_\epsilon
    \]
    Define the stopping times
    \[
    \tau_{\epsilon, a }:=\inf\{t\geq 0 :\ Z_\epsilon(t)\leq a_\epsilon\}, \tau_{\epsilon,b}:=\inf\{t\geq 0 :\ Z_\epsilon(t)\geq b_\epsilon\}. 
    \]
    Assume that 
    \begin{equation}\label{eq:star} 
       \Pb(\tau_{0,a_0}=+\infty \ \text{or } \ \forall \delta>0 \ \exists t_\pm\in(\tau_{0,a_0}, \tau_{0,a}+\delta) :\ Z_0(t_+)>a_0, \ Z_0(t_-)<a_0)=1, 
    \end{equation}
    and 
    \begin{equation}\label{eq:twostar} 
    \Pb(\tau_{0,b_0}=+\infty \ \text{or } \ \forall \delta>0 \ \exists t_\pm\in(\tau_{0,b_0}, \tau_{0,b_0}+\delta) :\ Z_0(t_+)>b_0, \ Z_0(t_-)<b_0)=1.
    \end{equation}
    Then we have the convergence in distribution 
    \[    Z_\epsilon(\cdot\wedge\tau_{\epsilon,a_\epsilon}\wedge\tau_{\epsilon,b_\epsilon}), \tau_{\epsilon,a_\epsilon}, \tau_{\epsilon,b_\epsilon})\rightarrow (Z_0(\cdot\wedge\tau_{0,a_0}\wedge\tau_{0,b_0}), \tau_{0,a_0},  \tau_{0,b_0})
    \]
    in the space $D([0,\infty))\times(\mbR\cup\{+\infty\})^2.$
\end{lemma}
The proof follows easily from Skorokhod's representation theorem.

Note that the Brownian motion with the hard membrane  satisfies the conditions \eqref{eq:star} and \eqref{eq:twostar} of the Lemma for any $a_0,b_0\neq 0$, since the Brownian motion itself satisfies these conditions.

Let us introduce a Brownian motion with a hard membrane at 0 and reflection at $\pm1$ as a process taking values in $[-1,1]$ and having the following properties. Initially, it behaves as a Brownian motion with a hard membrane until it exits the interval $(-\frac{2}{3}, \frac{2}{3}).$ Then it continues, independently of the past, as a Brownian motion with reflection at $-1$ (reflection upward) or 1 (reflection downward) until it enters the interval $[-\frac{1}{3}, \frac{1}{3}]$. After that, it again behaves as a Brownian motion with a hard membrane until it exits   $(-\frac{2}{3}, \frac{2}{3}),$ and so on. The following result is a consequence of Lemma \ref{lem:cont_tau} and Theorem \ref{th: Pilip1}.
\begin{theorem}
      Let $\tilde X_\epsilon$ be a solution to the SDE \eqref{eq:prelimit hard 1},
where
\[
\overline{\chi}_{\epsilon}\left(  x\right)  =\chi_{\epsilon,bdr}\left(
x-1\right)  \chi_{\epsilon,bdr}\left(  x+1\right)  \chi_{\epsilon}\left(
x\right),
\]  
$ {\chi}_{\epsilon}\left(  x\right)$ is an even function defined by \eqref{eq:chi_e}, 
% such that
%     \[
%     \chi_{\epsilon}\left(  x\right)  =\left\{
% \begin{array}
% [c]{ccc}%
% \frac{1}{2}\left(  \frac{x}{\epsilon}\right)  ^{\alpha} & \text{for} &
% x\in\left[  0,\epsilon\right), \\
% -\frac{1}{2}\frac{\left\vert x-2\epsilon\right\vert ^{\alpha}}{\epsilon
% ^{\alpha}}+1 & \text{for} & x\in\left[  \epsilon,2\epsilon\right),  \\
% 1 & \text{for} & x\geq2\epsilon,
% \end{array}
% \right.
%     \]
 $\chi_{\epsilon,bdr}$ {is equal to $\chi_\epsilon$}.
 
 If $\kappa_T^\epsilon$, $\kappa_\epsilon$, and $\alpha>0$ satisfy  
 the conditions of Theorem \ref{th:concergence without boundaties}, either in case of 1(a) or 2(b), then 
 \begin{equation}\label{eq:main_conv}
 (\tilde X_\epsilon(\cdot\wedge\tau_{\epsilon,-1}\wedge\tau_{\epsilon,1}), \tau_{\epsilon,-1}, \tau_{\epsilon,1})\Rightarrow
 ( X(\cdot\wedge\tau_{-1}\wedge\tau_{ 1}), \sigma_{ -1}, \sigma_{ 1}),
 \end{equation}
 where $X$ is a Brownian motion with a hard membrane and permeability parameters defined in 1(a) or 2(b), respectively, with refection at the points $- 1$ and $1$ (up and down, respectively),
\[\tau_{\epsilon,\pm1}=\inf\{t\geq 0\colon \tilde X_\epsilon(t)=\pm1\},\]
\[\sigma_{ \pm1}=\inf\{t\geq 0\colon l_{\pm1}(t)=\operatorname{Exp}_\pm(\beta)\},\]
$l_\pm$ is the symmetric local time of $X$ at $\pm1, $ respectively, and $\operatorname{Exp}_\pm(\beta)$ are independent exponentially distributed random variables with parameter $\beta$ that are also independent of $X$. 

If $\kappa_T^\epsilon$, $\kappa_\epsilon$, and $\alpha>0$ satisfy  
 the conditions of 1(b), 2(b), or 3 in Theorem \ref{th:concergence without boundaties}, then the convergence \eqref{eq:main_conv} holds with $X$ being a Brownian motion with refection at the points $-1$ and $1$.

In the case when $\kappa_T^\epsilon$, $\kappa_\epsilon$, and $\alpha>0$ satisfy  
 the conditions of 1(c) or 2(c), the process $X$ in \eqref{eq:main_conv} is a Brownian motion reflected at $0, -1$ and $1$.
\end{theorem}
%The proof follows from \ap{Theorem \ref{th: Pilip1}}. 
\vskip 5 pt

\begin{remark} Note that
\[ 
\lim_{\epsilon\to0}\Pb_x(\tau_{-1,\epsilon}<\tau_{\epsilon,1})=\Pb_x(\sigma_{-1}<\sigma_1).
\]
The probabilities $\Pb_x(\sigma_{-1} < \sigma_1)$ can be derived analogously to the proof of Theorem~\ref{thm:probab_hard}; we therefore state the corresponding result without proof.
         \begin{equation}\label{eq:P_tau_a_b1}
    \Pb_x(\sigma_{-1}<\sigma_1)=\begin{cases}
        -x\Pb_{(-1)+}(\sigma_{-1}<\sigma_1)+(1+x) \Pb_{0-}(\sigma_{-1}<\sigma_1),& x\in (-1,0),\\
         (1-x)\Pb_{0+}(\sigma_{-1}<\sigma_1)+ x\Pb_{1-}(\sigma_{-1}<\sigma_1),& x\in(0,1).
    \end{cases}
  \end{equation} 
    where 
     \begin{eqnarray}
     \Pb_{(-1)+}(\sigma_{-1}<\sigma_1)&=&  \Pb_{(-1)+}(\sigma_{-1}<\tau_{0-})+   \Pb_{(-1)+}(\sigma_{-1}>\tau_{0-})\Pb_{0-}(\sigma_{-1}<\sigma_1), \nonumber\\
    \Pb_{0-}(\sigma_{-1}<\sigma_1)&=&  \Pb_{0-}(\tau_{(-1)+}<\tau_{0+}) \Pb_{(-1)+}(\sigma_{-1}<\sigma_1)\\ & & \quad\quad+\Pb_{0-}(\tau_{(-1)+}>\tau_{0+}) \Pb_{0+}(\sigma_{-1}<\sigma_1), \nonumber\\
   \Pb_{0+}(\sigma_{-1}<\sigma_1)&=&    \Pb_{0+}(\tau_{0-}<\tau_{1-}) \Pb_{0-}(\sigma_{-1}<\sigma_1)\\ && \quad\quad+\Pb_{0+}(\tau_{0-}>\tau_1-) \Pb_{1-}(\sigma_{-1}<\sigma_1), \nonumber\\
  \Pb_{1-}(\sigma_{-1}<\sigma_1)&=& \Pb_{1-}(\tau_{0+}<\sigma_1) \Pb_{0+}(\sigma_{-1}<\sigma_1). \label{eq:systemLE}
    \end{eqnarray}
    Here $\tau_a=\inf\{t\geq 0: \tilde X_t=a\}$, the points $(-1)+, 1-$ are defined analogously to $0+, 0-$.

 It follows from  \eqref{eq:tau_1_tau_b}, \eqref{eq:probab_elastic} that
  \begin{eqnarray*}
     \Pb_{(-1)+}(\sigma_{-1}<\tau_0-)&=& 1- \Pb_{(-1)+}(\sigma_{-1}>\tau_0-)=1-\frac{1}{1+\beta}=\frac{\beta}{1+\beta}\\
   &=&1-\Pb_{0-}(\tau_{(-1)+}<\tau_{0+})=\Pb_{0-}(\tau_{(-1)+}>\tau_{0+})\\
   &=& 1-\Pb_{0+}(\tau_{0-}<\tau_{1-}) = \Pb_{0+}(\tau_{0-}>\tau_{1-})\\
     &=&\Pb_{1-}(\tau_{0+}>\sigma_1) =1-\Pb_{1-}(\tau_{0+}<\sigma_1).
    \end{eqnarray*}
Substituting this into the system of linear equations \eqref{eq:systemLE} yields the desired probabilities.
\end{remark}

\subsection{Link with Subsection 2.1}

In Section 2.1 we have introduced the heat flux $\phi_{q}^{\epsilon}$ in the
stationary regime and, in the particular case $\alpha=1$ and
\[
\kappa_{\epsilon}=\left(  K\epsilon\right)  ^{2}\kappa_{T}^{\epsilon}%
\]
we have proved heuristically the (approximate) formula%
\[
\phi_{q}^{\epsilon}=\kappa_{T}^{\epsilon}K\frac{2}{2K+\pi}\frac{T^{+}}{2}.
\]
The coefficient
\[
c_{\epsilon}:=\kappa_{T}^{\epsilon}K\frac{2}{2K+\pi}%
\]
is like an \textit{effective thermal conductivity coefficient}: it is the
thermal conductivity of an homogeneous system which has the same heat flow as
our system with a barrier. 

In this section we have introduced the permeability parameters $\beta
^{+}=\beta^{-}$ of the symmetric hard membrane process. Recall, moreover, that
the hard membrane process is the limit of the process  $\widetilde{X}%
_{\epsilon}$ obtained from the original process $X_{\epsilon}$ by means of a
time change. Our aim in this final subsection is to recover heuristically the
result of Subsection 2.1 from the formulae proved in this Section.

We have proved above, for the hard membrane process, that
\[
\Pb_{0-}\left(  \tau_{1}<\tau_{-1}\right)  =\frac{\beta^{-}}{2\beta^{-}+1}.
\]
This is a Markov chain probability of heat flow from the left of the barrier
to the right hand, complementary to the probability of heat flow from the left
of the barrier to the right hand, given by $1-\frac{\beta^{-}}{2\beta^{-}+1}$.
This probability, multiplied by the speed of the process, is equal to
$c_{\epsilon}$, since it represents the percentage of heat flowing to the
right in the unit of time. Namely, if $v_{\epsilon}$ is the speed of the
process, we have%

\begin{equation}
c_{\epsilon}=v_{\epsilon}\frac{\beta^{-}}{2\beta^{-}+1}.\label{link}%
\end{equation}

Recall from Theorem 10 that (by a simple computation of the integral)
\[
\beta^{+}=\frac{K}{\pi}.
\]
Then%
\[
\frac{\beta^{-}}{2\beta^{-}+1}=\frac{K}{2K^{-}+\pi}.
\]
Moreover, recall that the time change is%
\begin{align*}
\beta_{t}^{\epsilon}  & =2\kappa_{T}^{\epsilon}\int_{0}^{t}\left(
\frac{\kappa_{\epsilon}}{\kappa_{T}^{\epsilon}}+\overline{\chi}_{\epsilon}%
^{2}\left(  X_{\epsilon}\left(  s\right)  \right)  \right)  ds\\
& \sim2\kappa_{T}^{\epsilon}t.
\end{align*}
hence, in the limit as $\epsilon\rightarrow0$, it is naturally approximated
by
\[
\beta_{t}^{\epsilon}\sim2\kappa_{T}^{\epsilon}t.
\]
In other words, we have $v_{\epsilon}=2\kappa_{T}^{\epsilon}$. Hence
$v_{\epsilon}\frac{\beta^{-}}{2\beta^{-}+1}$ is equal to $c_{\epsilon}$.

\section{Appendix} \label{Append}

In this appendix we give the proof of Theorem \ref{Thm diffusion limit}. We divide it into several
steps. 

\textbf{Step 0}. Consider the Hilbert space $H=L^{2}\left(  D\right)  $.
Denote by $W^{k,p}\left(  D\right)  $ the Sobolev spaces with positive integer
differentiability degree $k$ and integrability power $p\geq1$ ($D^{\alpha}f\in
L^{p}\left(  D\right)  $ for every multi-index $\alpha$ such that $\left\vert
\alpha\right\vert =k$). Denote by $\mathcal{D}\left(  A_{\epsilon}\right)  $ the
subspaces of $L^{2}\left(  D\right)  $ defined as
\[
\mathcal{D}\left(  A_{\epsilon}\right)  =\left\{  f\in W^{2,2}\left(  D\right)
:f\left(  -1,y\right)  =f\left(  1,y\right)  =0\text{, }f\left(  x,0\right)
=f\left(  x,2\pi\right)  \right\}  .
\]
Assume that the coefficients $\overline{\chi}_{\epsilon}^{2}$ are smooth. Then
it is known (see \cite{AgrestiLuongo} sections A1-A2) that the operator
$A_{\epsilon}:\mathcal{D}\left( A_{\epsilon}\right)  \rightarrow L^{2}\left(  D\right)
$ defined as%
\[
A_{\epsilon}f=\kappa_{\epsilon}\Delta f+\kappa_{T}^{\epsilon}%
\operatorname{div}\left(  \overline{\chi}_{\epsilon}^{2}\nabla f\right)
\]
is the infinitesimal generator of a selfadjoint strongly continuous semigroup
(in fact an analytic semigroup; see \cite{Pazy} for general details),
$e^{tA_{\epsilon}}:H\rightarrow H$, $t\geq0$. If $\phi\in \mathcal{D}\left(
A_{\epsilon}\right)  $, then $u\left(  t\right)  =e^{tA_{\epsilon}}\phi$
belongs to $C\left(  0,T;\mathcal{D}\left(  A_{\epsilon}\right)  \right)  $ and to
$C^{1}\left(  0,T;H\right)  $, for every $T>0$,\ and solves uniquely the
Cauchy problem%
\[
\frac{du\left(  t\right)  }{dt}=A_\epsilon u\left(  t\right)  ,\qquad t\geq0,\qquad
u\left(  0\right)  =\phi.
\]
Moreover, if%
\[
\phi\in \mathcal{D}\left(  A_{\epsilon}^{2}\right)  =\left\{  f\in \mathcal{D}\left(  A_{\epsilon
}\right)  :A_{\epsilon}f\in \mathcal{D}\left(  A_{\epsilon}\right)  \right\}
\]%
\[
=\left\{  f\in W^{4,2}\left(  D\right)  \cap \mathcal{D}\left(  A_{\epsilon}\right)
:\left(  A_{\epsilon}f\right)  \left(  -1,y\right)  =\left(  A_{\epsilon
}f\right)  \left(  1,y\right)  \text{, }\left(  A_{\epsilon}f\right)  \left(
x,0\right)  =\left(  A_{\epsilon}f\right)  \left(  x,2\pi\right)  \right\}
\]
then $u\left(  t\right)  =e^{tA_{\epsilon}}\phi$ has the additional regularity
$C\left(  0,T;\mathcal{D}\left(  A_{\epsilon}^{2}\right)  \right)  $ and to
$C^{1}\left(  0,T;\mathcal{D}\left(  A_{\epsilon}\right)  \right)  $. Finally, by
Sobolev embedding in 2D we have $W^{4,2}\left(  D\right)  \subset C^{1}\left(
D\right)  $, hence also $\mathcal{D}\left(  A_{\epsilon}^{2}\right)  \subset
C^{1}\left(  D\right)  $, with continuous injection.

From these facts it follows that, if $\phi\in \mathcal{D}\left(  A_{\epsilon}%
^{2}\right)  $, then
\begin{equation}
\sup_{t\in\left[  0,T\right]  }\left\Vert \nabla e^{tA_{\epsilon}}%
\phi\right\Vert _{\infty}\leq C\left(  T,\phi\right)
\label{regularity inequality},%
\end{equation}
where the constant $C\left(  T,\phi\right)  $ depends on $T$ and $\left\Vert
A_{\epsilon}^{2}\phi\right\Vert _{H}+\left\Vert \phi\right\Vert _{H}$. The
inequality (\ref{regularity inequality}) certainly holds true in a larger
generality than the condition $\phi\in \mathcal{D}\left(  A_{\epsilon}^{2}\right)  $,
using interpolation spaces, but the arguments are longer and we do not need
them for the purpose of this work. A theory of fractional powers and
interpolation spaces is also developed in \cite{AgrestiLuongo}, see Theorem 2.1.

\textbf{Step 1}: precise form of the covariance function and the additional
elliptic operator. 

Above, we have introduced the covariance function
$Q_{N,\epsilon}\left(  \mathbf{x},\mathbf{x}^\prime\right)  $. It decomposes into the
addends $Q_{N,\epsilon}^{\left(  1\right)  }\left(  \mathbf{x},\mathbf{x}^\prime%
\right)  $, ..., $Q_{N,\epsilon}^{\left(  8\right)  }\left(  \mathbf{x}%
,\mathbf{x}^\prime\right)  $, where%
\begin{align}
Q_{N,\epsilon}^{\left(  1\right)  }\left(  \mathbf{x},\mathbf{x}^\prime\right)    \label{eq:Q1}&
=\frac{1}{2}\frac{1}{c_{N}^{2}}\sum_{\mathbf{k}\in K_{N}\cap\mathbb{Z}_{+}%
^{2}}\frac{\mathbf{k}^{\perp}\otimes\mathbf{k}^{\perp}}{\left\vert
\mathbf{k}\right\vert ^{2}}\overline{\chi}_{\epsilon}\left(  x\right)
\overline{\chi}_{\epsilon}\left(  x^{\prime}\right)  \cos\left(
\mathbf{k}\cdot\mathbf{x}\right)  \cos\left(  \mathbf{k}\cdot\mathbf{x}%
^{\prime}\right),  \\
Q_{N,\epsilon}^{\left(  2\right)  }\left(  \mathbf{x},\mathbf{x}^\prime\right)    \label{eq:Q2}&
=\frac{1}{2}\frac{1}{c_{N}^{2}}\sum_{\mathbf{k}\in K_{N}\cap\mathbb{Z}_{+}%
^{2}}\frac{\mathbf{k}^{\perp}\otimes\nabla^{\perp}\overline{\chi}_{\epsilon
}\left(  x^{\prime}\right)  }{\left\vert \mathbf{k}\right\vert ^{2}}%
\cos\left(  \mathbf{k}\cdot\mathbf{x}\right)  \sin\left(  \mathbf{k}%
\cdot\mathbf{x}^{\prime}\right)  \overline{\chi}_{\epsilon}\left(  x\right),
\\
Q_{N,\epsilon}^{\left(  3\right)  }\left(  \mathbf{x},\mathbf{x}^\prime\right)   &
=\frac{1}{2}\frac{1}{c_{N}^{2}}\sum_{\mathbf{k}\in K_{N}\cap\mathbb{Z}_{+}%
^{2}}\frac{\nabla^{\perp}\overline{\chi}_{\epsilon}\left(  x\right)
\otimes\mathbf{k}^{\perp}}{\left\vert \mathbf{k}\right\vert ^{2}}\sin\left(
\mathbf{k}\cdot\mathbf{x}\right)  \cos\left(  \mathbf{k}\cdot\mathbf{x}%
^{\prime}\right)  \overline{\chi}_{\epsilon}\left(  x^{\prime}\right),   \\
Q_{N,\epsilon}^{\left(  4\right)  }\left(  \mathbf{x},\mathbf{x}^\prime\right)    &
=\frac{1}{2}\frac{1}{c_{N}^{2}}\sum_{\mathbf{k}\in K_{N}\cap\mathbb{Z}_{+}%
^{2}}\frac{\nabla^{\perp}\overline{\chi}_{\epsilon}\left(  x\right)
\otimes\nabla^{\perp}\overline{\chi}_{\epsilon}\left(  x^{\prime}\right)
}{\left\vert \mathbf{k}\right\vert ^{2}}\sin\left(  \mathbf{k}\cdot
\mathbf{x}\right)  \sin\left(  \mathbf{k}\cdot\mathbf{x}^{\prime}\right), \label{eq:Q3}
\end{align}
and%
\begin{align}
Q_{N,\epsilon}^{\left(  5\right)  }\left(  \mathbf{x},\mathbf{x}^\prime\right)    &
=\frac{1}{2}\frac{1}{c_{N}^{2}}\sum_{\mathbf{k}\in K_{N}\cap\mathbb{Z}_{-}%
^{2}}\frac{\mathbf{k}^{\perp}\otimes\mathbf{k}^{\perp}}{\left\vert
\mathbf{k}\right\vert ^{2}}\overline{\chi}_{\epsilon}\left(  x\right)
\overline{\chi}_{\epsilon}\left(  x^{\prime}\right)  \sin\left(
\mathbf{k}\cdot\mathbf{x}\right)  \sin\left(  \mathbf{k}\cdot\mathbf{x}%
^{\prime}\right), \label{eq:Q5} \\
Q_{N,\epsilon}^{\left(  6\right)  }\left(  \mathbf{x},\mathbf{x}^\prime\right)    &
=-\frac{1}{2}\frac{1}{c_{N}^{2}}\sum_{\mathbf{k}\in K_{N}\cap\mathbb{Z}%
_{-}^{2}}\frac{\mathbf{k}^{\perp}\otimes\nabla^{\perp}\overline{\chi
}_{\epsilon}\left(  x^{\prime}\right)  }{\left\vert \mathbf{k}\right\vert
^{2}}\sin\left(  \mathbf{k}\cdot\mathbf{x}\right)  \overline{\chi}_{\epsilon
}\left(  x\right)  \cos\left(  \mathbf{k}\cdot\mathbf{x}^{\prime}\right), \label{eq:Q6} \\
Q_{N,\epsilon}^{\left(  7\right)  }\left(  \mathbf{x},\mathbf{x}^\prime\right)    &
=-\frac{1}{2}\frac{1}{c_{N}^{2}}\sum_{\mathbf{k}\in K_{N}\cap\mathbb{Z}%
_{-}^{2}}\frac{\nabla^{\perp}\overline{\chi}_{\epsilon}\left(  x\right)
\otimes\mathbf{k}^{\perp}}{\left\vert \mathbf{k}\right\vert ^{2}}\cos\left(
\mathbf{k}\cdot\mathbf{x}\right)  \sin\left(  \mathbf{k}\cdot\mathbf{x}%
^{\prime}\right)  \overline{\chi}_{\epsilon}\left(  x^{\prime}\right),  \label{eq:Q7} \\
Q_{N,\epsilon}^{\left(  8\right)  }\left(  \mathbf{x},\mathbf{x}^\prime\right)    &
=\frac{1}{2}\frac{1}{c_{N}^{2}}\sum_{\mathbf{k}\in K_{N}\cap\mathbb{Z}_{-}%
^{2}}\frac{\nabla^{\perp}\overline{\chi}_{\epsilon}\left(  x\right)
\otimes\nabla^{\perp}\overline{\chi}_{\epsilon}\left(  x^{\prime}\right)
}{\left\vert \mathbf{k}\right\vert ^{2}}\cos\left(  \mathbf{k}\cdot
\mathbf{x}\right)  \cos\left(  \mathbf{k}\cdot\mathbf{x}^{\prime}\right)  \label{eq:Q8}.
\end{align}
On the diagonal, we may summarize these terms as%
\[
Q_{N,\epsilon}\left(  \mathbf{x},\mathbf{x}\right)  =\overline{Q}_{N,\epsilon
}\left(  \mathbf{x},\mathbf{x}\right)  +R_{N,\epsilon}\left(  \mathbf{x}%
,\mathbf{x}\right),
\]%
where
\[
\overline{Q}_{N,\epsilon}\left(  \mathbf{x},\mathbf{x}\right)  =Q_{N,\epsilon
}^{\left(  1\right)  }\left(  \mathbf{x},\mathbf{x}\right)  +Q_{N,\epsilon
}^{\left(  5\right)  }\left(  \mathbf{x},\mathbf{x}\right)  =2\overline{\chi
}_{\epsilon}^{2}\left(  x\right)  I_{2},%
\]%
$I_{2}$ is the identity matrix in two dimensions,
\begin{align*}
R_{N,\epsilon}\left(  \mathbf{x},\mathbf{x}\right)    & =Q_{N,\epsilon
}^{\left(  2\right)  }\left(  \mathbf{x},\mathbf{x}\right)  +Q_{N,\epsilon
}^{\left(  3\right)  }\left(  \mathbf{x},\mathbf{x}\right)  +Q_{N,\epsilon
}^{\left(  4\right)  }\left(  \mathbf{x},\mathbf{x}\right)  \\
& +Q_{N,\epsilon}^{\left(  6\right)  }\left(  \mathbf{x},\mathbf{x}\right)
+Q_{N,\epsilon}^{\left(  7\right)  }\left(  \mathbf{x},\mathbf{x}\right)
+Q_{N,\epsilon}^{\left(  8\right)  }\left(  \mathbf{x},\mathbf{x}\right)  .
\end{align*}
Important is the fact that $\|Q_{N,\varepsilon}\|_{\infty}$ is bounded above by a constant independent of $N$ (and also of~$\varepsilon$). Indeed, this property holds for $\overline{Q}_{N,\epsilon}\left(  \mathbf{x},\mathbf{x}\right)  $ and, even more
easily, for $Q_{N,\epsilon}\left(  \mathbf{x},\mathbf{x}\right)  $.

\textbf{Step 2}: reduction to homogeneous boundary conditions and mild
formulation. 

The reduction to homogeneous boundary conditions, necessary for
the implementation of the semigroup approach, is simply based on the
subtraction of an interpolation of the boundary conditions. Called
$\gamma\left(  x,y\right)  $ the function in $D$ defined as
\[
\gamma\left(  x,y\right)  =T_{+}\frac{1-x}{2}%
\]
which is equal to $T_{+}$ at $x=-1$ and to zero at $x=1$, we set
\[
\widetilde{T}_{N,\epsilon}\left(  x,y,t\right)  =T_{N,\epsilon}\left(
x,y,t\right)  -\gamma\left(  x,y\right)
\]
so that%
\[
\widetilde{T}_{N,\epsilon}\left(  -1,y,t\right)  =\widetilde{T}_{N,\epsilon
}\left(  1,y,t\right)  =0
\]
and $\widetilde{T}_{N,\epsilon}$ is periodic in the vertical direction. We now
have%
\begin{align*}
& d\widetilde{T}_{N,\epsilon}+\sqrt{\kappa_{T}^{\epsilon}}\sum_{\mathbf{k}%
\in\mathbb{Z}_{0}^{2}}C\left(  \mathbf{k},N\right)  \mathbf{\sigma
}_{\mathbf{k}}^{\epsilon}\left(  x,y\right)  \cdot\nabla\left(  \widetilde{T}%
_{N,\epsilon}+\gamma\right)  dW_{t}^{\mathbf{k}}\\
& =\kappa_{\epsilon}\Delta\widetilde{T}_{N,\epsilon}dt+\kappa_{T}^{\epsilon
}\operatorname{div}\left(  Q_{N,\epsilon}\left(  \mathbf{x},\mathbf{x}\right)
\nabla\left(  \widetilde{T}_{N,\epsilon}+\gamma\right)  \right)  dt.
\end{align*}
Similarly, the limit equation, for the new variable %
\[
\widetilde{T}_{\epsilon}\left(  x,y,t\right)  =T_{\epsilon}\left(
x,y,t\right)  -\gamma\left(  x,y\right)
\]
reads%
\begin{align*}
\partial_{t}\widetilde{T}_{\epsilon}  & =\kappa_{\epsilon}\Delta
\widetilde{T}_{\epsilon}+\kappa_{T}^{\epsilon}\operatorname{div}\left(
\overline{\chi}_{\epsilon}^{2}\nabla\left(  \widetilde{T}_{\epsilon}%
+\gamma\right)  \right)  \\
& =A_{\epsilon}\widetilde{T}_{\epsilon}-\frac{T_{+}\kappa_{T}^{\epsilon}}%
{2}\partial_{x}\overline{\chi}_{\epsilon}^{2},%
\end{align*}
where the operator $A_\epsilon$ has been introduced in Step 0 above.

Setting
\begin{align*}
B_{N,\epsilon}f  & =\kappa_{T}^{\epsilon}\mathcal{L}_{N,\epsilon}f-\kappa
_{T}^{\epsilon}\operatorname{div}\left(  \overline{\chi}_{\epsilon}^{2}\nabla
f\right)  \\
& =\kappa_{T}^{\epsilon}\operatorname{div}\left(  R_{N,\epsilon}\left(
\mathbf{x},\mathbf{x}\right)  \nabla f\right),
\end{align*}
where $\mathcal{L}_{N,\epsilon}$ is defined in \eqref{eq:L},
we may also write the stochastic equation in the form
\begin{align*}
& d\widetilde{T}_{N,\epsilon}\\
& =\left(  A_{\epsilon}\widetilde{T}_{N,\epsilon}-\frac{T_{+}\kappa
_{T}^{\epsilon}}{2}\partial_{x}\overline{\chi}_{\epsilon}^{2}\right)  dt\\
& -\sqrt{\kappa_{T}^{\epsilon}}\sum_{\mathbf{k}\in\mathbb{Z}_{0}^{2}}C\left(
\mathbf{k},N\right)  \mathbf{\sigma}_{\mathbf{k}}^{\epsilon}\left(
x,y\right)  \cdot\nabla\left(  \widetilde{T}_{N,\epsilon}+\gamma\right)
dW_{t}^{\mathbf{k}}+B_{N,\epsilon}\left(  \widetilde{T}_{N,\epsilon}%
+\gamma\right)  dt
\end{align*}
which emphasises the similarities and differences with respect to the limit
equation. Denoting by $e^{tA_{\epsilon}}$ the analytic semigroup generated by
$A_{\epsilon}$, we may write the mild formulation of both equations %
\begin{align*}
\widetilde{T}_{N,\epsilon}\left(  t\right)    & =e^{tA_{\epsilon}%
}\widetilde{T}_{N,\epsilon}\left(  0\right)  -\frac{T_{+}\kappa_{T}^{\epsilon}%
}{2}\int_{0}^{t}e^{\left(  t-s\right)  A_{\epsilon}}\partial_{x}\overline
{\chi}_{\epsilon}^{2}ds\\
& -\sqrt{\kappa_{T}^{\epsilon}}\sum_{\mathbf{k}\in\mathbb{Z}_{0}^{2}}C\left(
\mathbf{k},N\right)  \int_{0}^{t}e^{\left(  t-s\right)  A_{\epsilon}%
}\mathbf{\sigma}_{\mathbf{k}}^{\epsilon}\cdot\nabla\left(  \widetilde{T}%
_{N,\epsilon}\left(  s\right)  +\gamma\right)  dW_{s}^{\mathbf{k}}\\
& +\int_{0}^{t}e^{\left(  t-s\right)  A_{\epsilon}}B_{N,\epsilon}\left(
\widetilde{T}_{N,\epsilon}\left(  s\right)  +\gamma\right)  ds
\end{align*}%
\[
\widetilde{T}_{\epsilon}\left(  t\right)  =e^{tA_{\epsilon}}\widetilde{T}%
_{\epsilon}\left(  0\right)  -\frac{T_{+}\kappa_{T}^{\epsilon}}{2}\int_{0}%
^{t}e^{\left(  t-s\right)  A_{\epsilon}}\partial_{x}\overline{\chi}_{\epsilon
}^{2}ds.
\]
Therefore, for the new variable %
\[
\widetilde{\theta}_{N,\epsilon}\left(  t\right)  :=\widetilde{T}_{N,\epsilon
}\left(  t\right)  -\widetilde{T}_{\epsilon}\left(  t\right)
\]
we have%
\begin{align*}
\widetilde{\theta}_{N,\epsilon}\left(  t\right)    & =-\sqrt{\kappa
_{T}^{\epsilon}}\sum_{\mathbf{k}\in\mathbb{Z}_{0}^{2}}C\left(  \mathbf{k}%
,N\right)  \int_{0}^{t}e^{\left(  t-s\right)  A_{\epsilon}}\mathbf{\sigma
}_{\mathbf{k}}^{\epsilon}\cdot\nabla\left(  \widetilde{T}_{N,\epsilon}\left(
s\right)  +\gamma\right)  dW_{s}^{\mathbf{k}}\\
& +\int_{0}^{t}e^{\left(  t-s\right)  A_{\epsilon}}B_{N,\epsilon}\left(
\widetilde{T}_{N,\epsilon}\left(  s\right)  +\gamma\right)  ds.
\end{align*}
We have to prove that, in the weak sense, this function goes to zero as
$N\rightarrow\infty$, in mean square sense.

\textbf{Step 3}: energy type estimate. From It\^{o} formula we get %
\begin{align*}
d\left\Vert \widetilde{T}_{N,\epsilon}\left(  t\right)  \right\Vert _{L^{2}%
}^{2}  & =-2\sqrt{\kappa_{T}^{\epsilon}}\sum_{\mathbf{k}\in\mathbb{Z}_{0}^{2}%
}C\left(  \mathbf{k},N\right)  \left\langle \mathbf{\sigma}_{\mathbf{k}%
}^{\epsilon}\cdot\nabla\left(  \widetilde{T}_{N,\epsilon}+\gamma\right)
,\widetilde{T}_{N,\epsilon}\right\rangle dW_{t}^{\mathbf{k}}\\
& +2\kappa_{\epsilon}\left\langle \Delta\widetilde{T}_{N,\epsilon
},\widetilde{T}_{N,\epsilon}\right\rangle dt+2\kappa_{T}^{\epsilon
}\left\langle \operatorname{div}\left(  Q_{N,\epsilon}\nabla\left(
\widetilde{T}_{N,\epsilon}+\gamma\right)  \right)  ,\widetilde{T}_{N,\epsilon
}\right\rangle dt\\
& +\kappa_{T}^{\epsilon}\sum_{\mathbf{k}\in\mathbb{Z}_{0}^{2}}C^{2}\left(
\mathbf{k},N\right)  \left\langle \mathbf{\sigma}_{\mathbf{k}}^{\epsilon}%
\cdot\nabla\left(  \widetilde{T}_{N,\epsilon}+\gamma\right)  ,\mathbf{\sigma
}_{\mathbf{k}}^{\epsilon} \cdot\nabla\left(  \widetilde{T}%
_{N,\epsilon}+\gamma\right)  \right\rangle dt.
\end{align*}
Thanks to $\operatorname{div}\mathbf{\sigma}_{\mathbf{k}}^{\epsilon}=0$, we have
\[
\left\langle \mathbf{\sigma}_{\mathbf{k}}^{\epsilon}\cdot\nabla\widetilde{T}%
_{N,\epsilon},\widetilde{T}_{N,\epsilon}\right\rangle =\frac{1}{2}\int%
_{D}\mathbf{\sigma}_{\mathbf{k}}^{\epsilon}\cdot\nabla\widetilde{T}%
_{N,\epsilon}^{2}d\mathbf{x}=-\frac{1}{2}\int_{D}\widetilde{T}_{N,\epsilon}%
^{2}\operatorname{div}\mathbf{\sigma}_{\mathbf{k}}^{\epsilon}d\mathbf{x}=0.
\]
From the
definition of $Q_{N,\epsilon}$  and its
positivity, we get %
\begin{align*}
& d\left\Vert \widetilde{T}_{N,\epsilon}\left(  t\right)  \right\Vert _{L^{2}%
}^{2}+2\kappa_{\epsilon}\left\Vert \nabla\widetilde{T}_{N,\epsilon}\right\Vert
_{L^{2}}^{2}dt\\
& =-2\sqrt{\kappa_{T}^{\epsilon}}\sum_{\mathbf{k}\in\mathbb{Z}_{0}^{2}%
}C\left(  \mathbf{k},N\right)  \left\langle \mathbf{\sigma}_{\mathbf{k}%
}^{\epsilon}\cdot\nabla\gamma,\widetilde{T}_{N,\epsilon}\right\rangle
dW_{t}^{\mathbf{k}}-2\kappa_{T}^{\epsilon}\left\langle Q_{N,\epsilon}\nabla\left(
\widetilde{T}_{N,\epsilon}+\gamma\right)  ,\nabla\widetilde{T}_{N,\epsilon
}  \right\rangle dt\\
& +2\kappa_{T}^{\epsilon}\left\langle Q_{N,\epsilon}\nabla\left(
\widetilde{T}_{N,\epsilon}+\gamma\right)  ,\nabla\left(  \widetilde{T}_{N,\epsilon}+\gamma\right)  \right\rangle dt
\\
& {=}-2\sqrt{\kappa_{T}^{\epsilon}}\sum_{\mathbf{k}\in\mathbb{Z}_{0}^{2}%
}C\left(  \mathbf{k},N\right)  \left\langle \mathbf{\sigma}_{\mathbf{k}%
}^{\epsilon}\cdot\nabla\gamma,\widetilde{T}_{N,\epsilon}\right\rangle
dW_{t}^{\mathbf{k}}+2\kappa_{T}^{\epsilon}\left\langle Q_{N,\epsilon}\nabla\left(
\widetilde{T}_{N,\epsilon}+\gamma\right)  ,\nabla\gamma\right\rangle dt\\ 
& {=}-2\sqrt{\kappa_{T}^{\epsilon}}\sum_{\mathbf{k}\in\mathbb{Z}_{0}^{2}%
}C\left(  \mathbf{k},N\right)  \left\langle \mathbf{\sigma}_{\mathbf{k}%
}^{\epsilon}\cdot\nabla\gamma,\widetilde{T}_{N,\epsilon}\right\rangle
dW_{t}^{\mathbf{k}}-2\kappa_{T}^{\epsilon}\left\langle \operatorname{div}\left(Q_{N,\epsilon}%
\nabla\left(  \widetilde{T}_{N,\epsilon}+\gamma\right) \right), \gamma\right\rangle dt.
\end{align*}
Since%
\[
-2\kappa_{T}^{\epsilon}\left\langle \operatorname{div}\left(  Q_{N,\epsilon
}\nabla\left(  \widetilde{T}_{N,\epsilon}+\gamma\right)  \right)
,\gamma\right\rangle =2\kappa_{T}^{\epsilon}\left\langle Q_{N,\epsilon}%
\nabla\left(  \widetilde{T}_{N,\epsilon}+\gamma\right)  ,\nabla\gamma
\right\rangle
\]%
\begin{align*}
&  \leq2\kappa_{T}^{\epsilon}\left\Vert Q_{N,\epsilon}\right\Vert _{\infty
}\frac{\sqrt{\kappa_{\epsilon}}}{\sqrt{\kappa_{\epsilon}}}\left\Vert
\nabla\widetilde{T}_{N,\epsilon}\right\Vert _{L^{2}}\left\Vert \nabla
\gamma\right\Vert _{L^{2}}+2\kappa_{T}^{\epsilon}\left\Vert Q_{N,\epsilon
}\right\Vert _{\infty}\left\Vert \nabla\gamma\right\Vert _{L^{2}}^{2}\\
&  \leq\kappa_{\epsilon}\left\Vert \nabla\widetilde{T}_{N,\epsilon}\right\Vert
_{L^{2}}^{2}+\frac{1}{\kappa_{\epsilon}}\left(  \kappa_{T}^{\epsilon
}\left\Vert Q_{N,\epsilon}\right\Vert _{\infty}\left\Vert \nabla
\gamma\right\Vert _{L^{2}}\right)  ^{2}+2\kappa_{T}^{\epsilon}\left\Vert
Q_{N,\epsilon}\right\Vert _{\infty}\left\Vert \nabla\gamma\right\Vert _{L^{2}%
}^{2}%
\end{align*}
we get
\begin{align*}
& d\left\Vert \widetilde{T}_{N,\epsilon}\left(  t\right)  \right\Vert _{L^{2}%
}^{2}+\kappa_{\epsilon}\left\Vert \nabla\widetilde{T}_{N,\epsilon}\right\Vert
_{L^{2}}^{2}dt\\
& \leq-2\sqrt{\kappa_{T}^{\epsilon}}\sum_{\mathbf{k}\in\mathbb{Z}_{0}^{2}%
}C\left(  \mathbf{k},N\right)  \left\langle \mathbf{\sigma}_{\mathbf{k}%
}^{\epsilon}\cdot\nabla\gamma,\widetilde{T}_{N,\epsilon}\right\rangle
dW_{t}^{\mathbf{k}}\\
& +\left(  \frac{1}{\kappa_{\epsilon}}\left(  \kappa_{T}^{\epsilon}\left\Vert
Q_{N,\epsilon}\right\Vert _{\infty}\left\Vert \nabla\gamma\right\Vert _{L^{2}%
}\right)  ^{2}+2\kappa_{T}^{\epsilon}\left\Vert Q_{N,\epsilon}\right\Vert
_{\infty}\left\Vert \nabla\gamma
\right\Vert _{L^{2}}^2\right)  dt.
\end{align*}
Recall that $\epsilon>0$ is given, in this theorem; hence $\frac{1}%
{\kappa_{\epsilon}}$ is a given constant (it will diverge to infinity as
$\epsilon\rightarrow0$ in Section \ref{sec:limit hard} only). Therefore, from the boundedness of $\|Q_{N, \epsilon}\|_{\infty}$ we obtain %
\[
\mathbb{E}\left[  \left\Vert \widetilde{T}_{N,\epsilon}\left(  t\right)
\right\Vert _{L^{2}}^{2}\right]  +\kappa_{\epsilon}\mathbb{E}\int_{0}^{t}\left\Vert
\nabla\widetilde{T}_{N,\epsilon}\left(  s\right)  \right\Vert _{L^{2}}%
^{2}ds\leq\mathbb{E}\left[  \left\Vert \widetilde{T}_{\epsilon}\left(
0\right)  \right\Vert _{L^{2}}^{2}\right]  +C_{\epsilon}t
\]
for a suitable constant $C_{\epsilon}>0$. The precise classical strategy of
proof, that we omit, requires first to introduce a stopping time in order to
localize the It\^{o} integral reducing it to a martingale, make the estimate
in average on the stopped processes and finally use Fatou's lemma to remove the localization.

\textbf{Step 4}: convergence of the first element of $\widetilde{\theta
}_{N,\epsilon}$. The statements of the theorem are
equivalent to proving that
\[
\lim_{N\rightarrow\infty}\mathbb{E}\left[  \left\langle \widetilde{\theta
}_{N,\epsilon}\left(  t\right)  ,\phi\right\rangle ^{2}\right]  =0
\]
{for every $\phi\in \mathcal{D}\left( A_{\epsilon}^{2}\right)  $.}

\begin{proof}
First, recalling that $B_{N,\epsilon}f=\kappa_{T}^{\epsilon}\operatorname{div}%
\left(  R_{N,\epsilon}\nabla f\right)  $, we have%
\[
\left\langle \int_{0}^{t}e^{\left(  t-s\right)  A_{\epsilon}}B_{N, \epsilon
}\left(  \widetilde{T}_{N,\epsilon}\left(  s\right)  +\gamma\right)
ds,\phi\right\rangle =-\kappa_{T}^{\epsilon}\int_{0}^{t}\left\langle
R_{N,\epsilon}\nabla\left(  \widetilde{T}_{N,\epsilon}\left(  s\right)
+\gamma\right)  ,\nabla e^{\left(  t-s\right)  A_{\epsilon}}\phi\right\rangle
ds
\]%
\[
\leq\left\Vert R_{N,\epsilon}\right\Vert _{\infty}\sup_{s\in\left[
0,t\right]  }\left\Vert \nabla e^{\left(  t-s\right)  A_{\epsilon}}%
\phi\right\Vert _{L^{2}}\left(  \int_{0}^{t}\left\Vert \nabla\widetilde{T}%
_{N,\epsilon}\left(  s\right)  \right\Vert _{L^{2}}ds+t\left\Vert \nabla
\gamma\right\Vert _{L^{2}}\right)
\]
hence (using the estimate of Step 3)%
\begin{align*}
& \mathbb{E}\left[  \left\langle \int_{0}^{t}e^{\left(  t-s\right)
A_{\epsilon}}B_{N, \epsilon}\left(  \widetilde{T}_{N,\epsilon}\left(  s\right)
+\gamma\right)  ds,\phi\right\rangle ^{2}\right]  \\
& \leq 2\left\Vert R_{N,\epsilon}\right\Vert _{\infty}^{2}\sup_{s\in\left[
0,t\right]  }\left\Vert \nabla e^{\left(  t-s\right)  A_{\epsilon}}%
\phi\right\Vert _{L^{2}}^{2}t\left(  \frac{1}{\kappa_{\epsilon}}%
\mathbb{E}\left[  \left\Vert \widetilde{T}_{\epsilon}\left(  0\right)
\right\Vert _{L^{2}}^{2}\right]  +t\left\Vert \nabla\gamma\right\Vert
_{L^{2}}^{2}+ \frac{C_\epsilon}{\kappa_\epsilon}\right)  .
\end{align*}
{If $\phi\in \mathcal{D}\left( A_{\epsilon}^{2}\right)  $,}\ then $\sup_{s\in\left[  0,t\right]
}\left\Vert e^{\left(  t-s\right)  A_{\epsilon}}\phi\right\Vert _{W^{1,2}}%
^{2}$ is finite {by
property (\ref{regularity inequality}) (which is stronger than what is
required here).} Therefore, it is sufficient to prove that%
\[
\lim_{N\rightarrow\infty}\left\Vert R_{N,\epsilon}\right\Vert _{\infty}^{2}=0.
\]
Inspection into the definition of $R_{N,\epsilon}$\ shows that the terms can be estimated from above as follows%
\[
\frac{1}{c_{N}^{2}}\sum_{\mathbf{k}\in K_{N}\cap\mathbb{Z}_{+}^{2}}\left(
\frac{\mathbf{1}}{\left\vert \mathbf{k}\right\vert }+\frac{\mathbf{1}%
}{\left\vert \mathbf{k}\right\vert ^{2}}\right)  \leq\frac{1}{c_{N}^{2}%
}Card\left(  K_{N}\cap\mathbb{Z}_{+}^{2}\right)  \left(  \frac{\mathbf{1}}%
{N}+\frac{\mathbf{1}}{N^{2}}\right)
\]
which is infinitesimal as $N\rightarrow\infty$, by the definition of $c_{N}$.
This proves the convergence to zero of the first term.
\end{proof}

\textbf{Step 5}: convergence; second element. It remains to prove that
\begin{equation}\label{eq:expect_stoch_integral}
  \lim_{N\rightarrow\infty}\mathbb{E}\left[  \left(  \sum_{\mathbf{k}%
\in\mathbb{Z}_{0}^{2}}C\left(  \mathbf{k},N\right)  \int_{0}^{t}\left\langle
e^{\left(  t-s\right)  A_{\epsilon}}\mathbf{\sigma}_{\mathbf{k}}^{\epsilon
}\cdot\nabla\left(  \widetilde{T}_{N,\epsilon}\left(  s\right)  +\gamma
\right)  ,\phi\right\rangle dW_{s}^{\mathbf{k}}\right)  ^{2}\right]  =0.
\end{equation}

We have
\[
\left\langle e^{\left(  t-s\right)  A_{\epsilon}}\varphi,\phi\right\rangle
=\left\langle \varphi,e^{\left(  t-s\right)  A_{\epsilon}^{\ast}}%
\phi\right\rangle =\left\langle \varphi,e^{\left(  t-s\right)  A_{\epsilon}%
}\phi\right\rangle.
\]
since $A_{\epsilon}$ is selfadjoint. Hence, by the isometry formula, the expected value on the right-hand side of \eqref{eq:expect_stoch_integral} is equal to
\begin{align*}
&  \sum_{\mathbf{k}\in\mathbb{Z}_{0}^{2}}C^{2}\left(  \mathbf{k},N\right)
\int_{0}^{t}\mathbb{E}\left[  \left\langle e^{\left(  t-s\right)  A_{\epsilon
}}\mathbf{\sigma}_{\mathbf{k}}^{\epsilon}\cdot\nabla\left(  \widetilde{T}%
_{N,\epsilon}\left(  s\right)  +\gamma\right)  ,\phi\right\rangle ^{2}\right]
ds\\
&  =\sum_{\mathbf{k}\in\mathbb{Z}_{0}^{2}}C^{2}\left(  \mathbf{k},N\right)
\int_{0}^{t}\mathbb{E}\left[  \left\langle \mathbf{\sigma}_{\mathbf{k}%
}^{\epsilon}\cdot\nabla\left(  \widetilde{T}_{N,\epsilon}\left(  s\right)
+\gamma\right)  ,e^{\left(  t-s\right)  A_{\epsilon}}\phi\right\rangle
^{2}\right]  ds\\
& =\sum_{\mathbf{k}\in\mathbb{Z}_{0}^{2}}C^{2}\left(  \mathbf{k},N\right)
\int_{0}^{t}\mathbb{E}\left[  \left\langle \widetilde{T}_{N,\epsilon}\left(
s\right)  +\gamma,\mathbf{\sigma}_{\mathbf{k}}^{\epsilon}\cdot\nabla
e^{\left(  t-s\right)  A_{\epsilon}}\phi\right\rangle ^{2}\right]  ds.
\end{align*}
Here we use the general identity%
\[
\left\langle \mathbf{\sigma}_{\mathbf{k}}^{\epsilon}\cdot\nabla
f,g\right\rangle =-\left\langle f,\mathbf{\sigma}_{\mathbf{k}}^{\epsilon}%
\cdot\nabla g\right\rangle
\]
which easily follows from $\operatorname{div}\mathbf{\sigma}_{\mathbf{k}%
}^{\epsilon}=0$. 
Let us define  $\mathbf{v}%
_{t}\left(  \mathbf{x},s\right)  =\left(  \nabla e^{\left(  t-s\right)
A_{\epsilon}}\phi\right)  \left(  \mathbf{x}\right)  $ and recall that ${T}_{N,\epsilon}\left(
\mathbf{x},s\right)  =\widetilde{T}_{N,\epsilon}\left(
\mathbf{x},s\right)  +\gamma\left(  \mathbf{x}\right)  $. Then %
\begin{align*}
& \sum_{\mathbf{k}\in\mathbb{Z}_{0}^{2}}C^{2}\left(  \mathbf{k},N\right)
\left\langle \widetilde{T}_{N,\epsilon}\left(  s\right)  +\gamma
,\mathbf{\sigma}_{\mathbf{k}}^{\epsilon}\cdot\nabla e^{\left(  t-s\right)
A_{\epsilon}}\phi\right\rangle ^{2}\\
& =\sum_{\mathbf{k}\in\mathbb{Z}_{0}^{2}}C^{2}\left(  \mathbf{k},N\right)
\int\int {T}_{N,\epsilon}\left( \mathbf{x},s\right)  {T}_{N,\epsilon}\left( \mathbf{x}^{\prime},s\right)
\mathbf{\sigma}_{\mathbf{k}}^{\epsilon}\left(  \mathbf{x}\right)
\cdot\mathbf{v}_{t}\left(  \mathbf{x},s\right)  \quad\mathbf{\sigma
}_{\mathbf{k}}^{\epsilon}\left(  \mathbf{x}^{\prime}\right)  \cdot
\mathbf{v}_{t}\left(  \mathbf{x}^{\prime},s\right)  dxdx^{\prime}\\
& =2\int\int {T}_{N,\epsilon}\left( \mathbf{x},s\right)  {T}_{N,\epsilon}\left( \mathbf{x}^{\prime
},s\right)  \mathbf{v}_{t}\left(  \mathbf{x},s\right)  ^{T}Q_{N,\epsilon
}\left(  \mathbf{x},\mathbf{x}^{\prime}\right)  \mathbf{v}_{t}\left(
\mathbf{x}^{\prime},s\right)  dxdx^{\prime}.
\end{align*}
Introducing the linear bounded operator on $L^{2}\left(  D,\mathbb{R}%
^{2}\right)  $ given by %
\[
\left(  \mathbb{Q}_{N,\epsilon}\mathbf{v}\right)  \left(  \mathbf{x}\right)
=\int Q_{N,\epsilon}\left(  \mathbf{x},\mathbf{x}^{\prime}\right)
\mathbf{v}\left(  \mathbf{x}^{\prime}\right)  d\mathbf{x^{\prime}}%
\]
and the best constant $\delta\left(  N,\epsilon\right)  $ such that
\[
\left\langle \mathbb{Q}_{N,\epsilon}\mathbf{v},\mathbf{v}\right\rangle
_{L^{2}\left(  D,\mathbb{R}^{2}\right)  }\leq\delta\left(  N,\epsilon\right)
\left\Vert \mathbf{v}\right\Vert _{L^{2}\left(  D,\mathbb{R}^{2}\right)  }^{2}%
\]
we have proved until now that%
\begin{align*}
& \mathbb{E}\left[  \left(  \sum_{\mathbf{k}\in\mathbb{Z}_{0}^{2}}C\left(
\mathbf{k},N\right)  \int_{0}^{t}\left\langle e^{\left(  t-s\right)
A_{\epsilon}}\mathbf{\sigma}_{\mathbf{k}}^{\epsilon}\cdot\nabla\left(
\widetilde{T}_{N,\epsilon}\left(  s\right)  +\gamma\right)  ,\phi\right\rangle
dW_{s}^{\mathbf{k}}\right)  ^{2}\right]  \\
& =2\int_{0}^{t}\mathbb{E}\left\langle \mathbb{Q}_{N,\epsilon}\mathbf{v}%
_{t}\left(  s\right)  {T}_{N,\epsilon}\left( s\right)  ,\mathbf{v}_{t}\left(  s\right)
{T}_{N,\epsilon}\left( s\right)  \right\rangle ds\\
& \leq2\delta\left(  N,\epsilon\right)  \int_{0}^{t}\mathbb{E}\left[
\left\Vert \mathbf{v}_{t}\left(  s\right)  {T}_{N,\epsilon}\left( s\right)  \right\Vert
_{L^{2}\left(  D,\mathbb{R}^{2}\right)  }^{2}\right]  ds\\
& \leq2\delta\left(  N,\epsilon\right)  \sup_{s\in\left[  0,t\right]
}\left\Vert \nabla e^{\left(  t-s\right)  A_{\epsilon}}\phi\right\Vert
_{\infty}^{2}\int_{0}^{t}\mathbb{E}\left[  2\left\Vert \widetilde{T}%
_{N,\epsilon}\left(  s\right)  \right\Vert _{L^{2}}^{2}+2\left\Vert
\gamma\right\Vert _{L^{2}}^{2}\right]  ds.
\end{align*}
Since, by (\ref{regularity inequality}), $\sup_{s\in\left[  0,t\right]  }\left\Vert \nabla
e^{\left(  t-s\right)  A_{\epsilon}}\phi\right\Vert _{\infty}^{2}$ is bounded,
the result will be proved if we show that
\begin{equation}
\lim_{N\rightarrow\infty}\delta\left(  N,\epsilon\right)
=0.\label{final item}%
\end{equation}

\textbf{Step 6}: proof of (\ref{final item}). Recall that
\begin{align*}
\mathbf{\sigma}_{\mathbf{k}}^{\epsilon}\left(  \mathbf{x}\right)   &
=\mathbf{\sigma}_{\mathbf{k}}^{\epsilon,0}\left(  \mathbf{x}\right)
+\frac{\nabla^{\perp}\overline{\chi}_{\epsilon}\left(  x\right)  }{\left\vert
\mathbf{k}\right\vert }\sin\left(  \mathbf{k}\cdot\mathbf{x}\right)
\qquad\text{for }\mathbf{k}\in\mathbb{Z}_{+}^{2}\\
\mathbf{\sigma}_{\mathbf{k}}^{\epsilon}\left(  \mathbf{x}\right)   &
=\mathbf{\sigma}_{\mathbf{k}}^{\epsilon,0}\left(  \mathbf{x}\right)
+\frac{\nabla^{\perp}\overline{\chi}_{\epsilon}\left(  x\right)  }{\left\vert
\mathbf{k}\right\vert }\cos\left(  \mathbf{k}\cdot\mathbf{x}\right)
\qquad\text{for }\mathbf{k}\in\mathbb{Z}_{-}^{2}%
\end{align*}
where%
\begin{align*}
\mathbf{\sigma}_{\mathbf{k}}^{\epsilon,0}\left(  \mathbf{x}\right)   &
=\overline{\chi}_{\epsilon}\left(  x\right)  \frac{\mathbf{k}^{\perp}%
}{\left\vert \mathbf{k}\right\vert }\cos\left(  \mathbf{k}\cdot\mathbf{x}%
\right)  \qquad\text{for }\mathbf{k}\in\mathbb{Z}_{+}^{2}\\
\mathbf{\sigma}_{\mathbf{k}}^{\epsilon,0}\left(  \mathbf{x}\right)   &
=-\overline{\chi}_{\epsilon}\left(  x\right)  \frac{\mathbf{k}^{\perp}%
}{\left\vert \mathbf{k}\right\vert }\sin\left(  \mathbf{k}\cdot\mathbf{x}%
\right)  \qquad\text{for }\mathbf{k}\in\mathbb{Z}_{-}^{2}.
\end{align*}
Therefore, we may introduce the covariance function
\[
Q_{N,\epsilon}^{0}\left(  \mathbf{x},\mathbf{x}^{\prime}\right)  =\frac
{1}{2c_{N}^{2}}\sum_{\mathbf{k}\in K_{N}}\mathbf{\sigma}_{\mathbf{k}%
}^{\epsilon,0}\left(  \mathbf{x}\right)  \otimes\mathbf{\sigma}_{\mathbf{k}%
}^{\epsilon,0}\left(  \mathbf{x}^{\prime}\right)
\]
and the associated operator
\[
\left(  \mathbb{Q}_{N,\epsilon}^{0}\mathbf{v}\right)  \left(  \mathbf{x}%
\right)  =\int Q_{N,\epsilon}^{0}\left(  \mathbf{x},\mathbf{x}^{\prime
}\right)  \mathbf{v}\left(  \mathbf{x}^{\prime}\right)  d\mathbf{x^{\prime}}%
\]
and write%
\[
\left\langle \mathbb{Q}_{N,\epsilon}\mathbf{v},\mathbf{v}\right\rangle
_{L^{2}\left(  D,\mathbb{R}^{2}\right)  }=\left\langle \mathbb{Q}_{N,\epsilon
}^{0}\mathbf{v},\mathbf{v}\right\rangle _{L^{2}\left(  D,\mathbb{R}%
^{2}\right)  }+\left\langle \mathbb{R}_{N,\epsilon}^{0}\mathbf{v}%
,\mathbf{v}\right\rangle _{L^{2}\left(  D,\mathbb{R}^{2}\right)  }%
\]
where the operator $\mathbb{R}_{N,\epsilon}^{0}$ is made of a few terms;\ let
us write only one of them as an example: one term of $\left\langle
\mathbb{R}_{N,\epsilon}^{0}\mathbf{v},\mathbf{v}\right\rangle _{L^{2}\left(
D,\mathbb{R}^{2}\right)  }$ is, for instance, \
\[
\int\int\frac{1}{2c_{N}^{2}}\sum_{\mathbf{k}\in K_{N}\cap\mathbb{Z}_{+}^{2}%
}\mathbf{v}\left(  \mathbf{x}\right)  ^{T}\left[  \mathbf{\sigma}_{\mathbf{k}%
}^{\epsilon,0}\left(  \mathbf{x}\right)  \otimes\frac{\nabla^{\perp}%
\overline{\chi}_{\epsilon}\left(  x\right)  }{\left\vert \mathbf{k}\right\vert
}\sin\left(  \mathbf{k}\cdot\mathbf{x}\right)  \right]  \mathbf{v}\left(
\mathbf{x}^{\prime}\right)  d\mathbf{x}d\mathbf{x^{\prime}}.
\]
This term is bounded above, up to a constant, by%
\[
\left(  \frac{1}{c_{N}^{2}}\sum_{\mathbf{k}\in K_{N}\cap\mathbb{Z}_{+}^{2}%
}\frac{1}{\left\vert \mathbf{k}\right\vert }\right)  \left\Vert \mathbf{v}%
\right\Vert _{L^{2}\left(  D,\mathbb{R}^{2}\right)  }^{2}\leq\frac{1}%
{c_{N}^{2}}Card\left(  K_{N}\cap\mathbb{Z}_{+}^{2}\right)  \frac{1}%
{N}\left\Vert \mathbf{v}\right\Vert _{L^{2}\left(  D,\mathbb{R}^{2}\right)
}^{2}%
\]
which is infinitesimal as $N\rightarrow\infty$. Therefore the contribution of
$\left\langle \mathbb{R}_{N,\epsilon}^{0}\mathbf{v},\mathbf{v}\right\rangle
_{L^{2}\left(  D,\mathbb{R}^{2}\right)  }$ to the constant $\delta\left(
N,\epsilon\right)  $ is infinitesimal. It remains to estimate $\left\langle
\mathbb{Q}_{N,\epsilon}^{0}\mathbf{v},\mathbf{v}\right\rangle _{L^{2}\left(
D,\mathbb{R}^{2}\right)  }$. Because the vector fields $\frac{\mathbf{k}^{\perp}%
}{\left\vert \mathbf{k}\right\vert }\cos\left(  \mathbf{k}\cdot\mathbf{x}%
\right)  $, $\frac{\mathbf{k}^{\perp}}{\left\vert \mathbf{k}\right\vert }%
\sin\left(  \mathbf{k}\cdot\mathbf{x}\right)  $ are a subset of a complete
orthonormal system $\left(  \mathbf{e}_{n}\right)  _{n\in\mathbb{N}}$ in
$L^{2}\left(  D,\mathbb{R}^{2}\right)  $, we get that
\begin{align*}
\left\langle \mathbb{Q}_{N,\epsilon}^{0}\mathbf{v},\mathbf{v}\right\rangle
_{L^{2}\left(  D,\mathbb{R}^{2}\right)  }  & =\sum_{\mathbf{k}\in
\mathbb{Z}_{0}^{2}}C^{2}\left(  \mathbf{k},N\right)  \left(  \int\sigma
_{k}^{\epsilon,0}\left(  \mathbf{x}\right)  \cdot\mathbf{v}\left(
\mathbf{x}\right)  dx\right)  ^{2}\\
& \leq\frac{1}{c_{N}^{2}}\sum_{n\in\mathbb{N}}\left\langle \overline{\chi
}_{\epsilon}\mathbf{v},\mathbf{e}_{n}\right\rangle ^{2}=\frac{1}{c_{N}^{2}%
}\left\Vert \overline{\chi}_{\epsilon}\mathbf{v}\right\Vert _{L^{2}\left(
D,\mathbb{R}^{2}\right)  }^{2}\leq\frac{1}{c_{N}^{2}}\left\Vert \overline
{\chi}_{\epsilon}\right\Vert _{\infty}\left\Vert \mathbf{v}\right\Vert
_{L^{2}\left(  D,\mathbb{R}^{2}\right)  }^{2},%
\end{align*}
where again $\frac{1}{c_{N}^{2}}\left\Vert \overline{\chi}_{\epsilon
}\right\Vert _{\infty}$ is infinitesimal as $N\rightarrow\infty$. The proof is complete.

\section{Acknowledgments}
{We thank Eliseo Luongo for useful comments about the regularity results of the Appendix.
}

The research has been funded by the European Union (ERC, NoisyFluid, No.
101053472). Views and opinions expressed are however those of the authors only
and do not necessarily reflect those of the European Union or the European
Research Council. Neither the European Union nor the granting authority can be
held responsible for them.

A.~P. thanks  the Swiss National Science Foundation for partial support  of the paper (grants No IZRIZ0\_226875, No 200020\_200400, No. 200020\_192129).

O.~A. was partially supported by Deutsche Forschungsgemeinschaft (Project No. 548113512).

\end{document}